\documentclass[a4paper,11pt]{article}
\usepackage{makeidx}
\usepackage{graphicx}
\usepackage[english,activeacute]{babel}
\usepackage[latin1]{inputenc}
\usepackage{amsmath}
\usepackage{amsfonts}
\usepackage{amssymb,latexsym}
\usepackage{colortbl}
\usepackage{multicol}
\usepackage{hyperref}
\usepackage{booktabs}

\newtheorem{theorem}{Theorem}

\newtheorem{corollary}{Corollary}

\newtheorem{lemma}{Lemma}

\newtheorem{proposition}{Proposition}
\newtheorem{remark}{Remark}

\newenvironment{proof}[1][Proof]{\noindent\textbf{#1.} }{\ \rule{0.5em}{0.5em}}
\begin{document}

\title{On Freud-Sobolev type orthogonal polynomials}
\author{Luis E. Garza$^{1}$, Edmundo J. Huertas$^{2,\dagger}$, and Francisco
Marcellán$^{3}$\thanks{%
Part of this research was conducted while L.E. Garza was visiting E.J. Huertas at the Universidad de Alcal\'{a} in early 2017, under the ``GINER DE LOS RIOS'' research program. Both authors wish to thank the Departamento de F\'{i}sica y Matem\'{a}ticas de la Universidad de Alcal\'{a} for its support. The work of the three authors was partially supported by Direcci\'{o}n General de Investigaci\'{o}n Cient\'{i}fica y T\'{e}cnica, Ministerio de Econom\'{i}a y Competitividad of Spain, under grant MTM2015-65888-C4-2-P. ($\dagger$) Corresponding author.} \\
\\
$^{1}$Facultad de Ciencias, Universidad de Colima,\\
Bernal Díaz del Castillo, No. 340, C.P. 28045 Colima, M\'exico.\\
luis\textunderscore garza1@ucol.mx, garzaleg@gmail.com\\
\\
$^{2}$Departamento de Ingeniería Civil: Hidráulica y Ordenación del
Territorio,\\
E.T.S. de Ingeniería Civil, Universidad Politécnica de Madrid,\\
C/ Alfonso XII, 3 y 5, 28014 Madrid, Spain.\\
ej.huertas.cejudo@upm.es, ehuertasce@gmail.com\\
\\
$^{3}$Instituto de Ciencias Matemáticas (ICMAT) and Departamento de
Matem\'aticas,\\
Universidad Carlos III de Madrid, Avenida de la Universidad 30, 28911, Legané%
s, Spain\\
pacomarc@ing.uc3m.es}
\date{\emph{(\today)}}
\maketitle

\begin{abstract}
In this contribution we deal with sequences of monic polynomials orthogonal
with respect to the Freud Sobolev-type inner product%
\begin{equation*}
\left\langle p,q\right\rangle _{s}=\int_{\mathbb{R}%
}p(x)q(x)e^{-x^{4}}dx+M_{0}p(0)q(0)+M_{1}p^{\prime }(0)q^{\prime }(0),
\end{equation*}%
where $p,q$ are polynomials, $M_{0}$ and $M_{1}$ are nonnegative real
numbers. Connection formulas between these polynomials and Freud polynomials
are deduced and, as a consequence, a five term recurrence relation for such
polynomials is obtained. The location of their zeros as well as their
asymptotic behavior is studied. Finally, an electrostatic interpretation of
them in terms of a logarithmic interaction in the presence of an external
field is given.

\smallskip


\end{abstract}



\section{Introduction}

\label{[Sec-1]-Intro}



Let us consider the so called Freud type inner products%
\begin{equation}
\left\langle p,q\right\rangle =\int_{\mathbb{R}}p(x)q(x)d\mu (x),\quad
p,q\in \mathbb{P},  \label{GeneralInnProd}
\end{equation}%
where $d\mu (x)=\omega (x)dx=e^{-V(x)}dx$ is a positive, nontrivial Borel
measure supported on the whole real line $\mathbb{R}$, and $\mathbb{P}$ is
the linear space \ of polynomials with real coefficients. Analytic
properties of such sequences of polynomials are very well known for certain
values of the external potential $V(x)$.

Let us introduce the following inner product in $\mathbb{P}$

\begin{equation}
\langle p,q\rangle =\int_{\mathbb{R}}p(x)q(x)e^{-x^{4}}dx,\quad p,q\in 
\mathbb{P},  \label{Fb-InnerProduct}
\end{equation}%
i.e., $V(x)=x^{4}$ in (\ref{GeneralInnProd}).

Let $\{F_{n}(x)\}_{n\geq 0}$ be the\ corresponding sequence of monic
orthogonal polynomials (MOPS, in short), which constitute a family of
semi--classical orthogonal polynomials (see \cite{M-ACAM91}, \cite{S-DMJ39}%
), because $V(x)$ is differentiable in $\mathbb{R}$\ (the support of $\mu $%
), and the linear functional $u$\ associated with $\omega (x)=e^{-V(x)}$,
i.e.%
\begin{equation*}
\langle u,p(x)\rangle =\int_{\mathbb{R}}p(x)\omega (x)dx,
\end{equation*}%
satisfies the distributional (or Pearson) equation (see \cite{V-WS07})%
\begin{equation*}
\lbrack \sigma (x)\omega (x)]^{\prime }=\tau (x)\omega (x),
\end{equation*}%
where $\sigma (x)=1$ and $\tau (x)=-4x^{3}$. Notice that, in terms of the
weight function, the above relation means that%
\begin{equation*}
\frac{\omega ^{\prime }(x)}{\omega (x)}=\frac{\tau (x)-\sigma ^{\prime }(x)}{%
\sigma (x)}=-V^{\prime }(x).
\end{equation*}

In this contribution, we consider the diagonal Freud Sobolev-type inner
product%
\begin{equation}
\left\langle p,q\right\rangle _{s}=\left\langle p,q\right\rangle +\mathbf{p}%
^{T}(0)\mathbf{Mq}(0),  \label{SobFbInnProd}
\end{equation}%
where%
\begin{equation*}
\mathbf{p}(0)=[p(0),p^{\prime }(0),\ldots ,p^{(s)}(0)]^{T}
\end{equation*}%
is a column vector of dimension $s+1$, the column vector $\mathbf{q}(0)$ is
defined in an analogous way, and $\mathbf{M}$ is the diagonal and positive
definite $(s+1)\times (s+1)$ matrix%
\begin{equation*}
\mathbf{M}=diag\,[M_{0},M_{1},\ldots ,M_{s}],\quad M_{k}\in \mathbb{R}_{+},
k=0,1,\cdots, s\text{.}
\end{equation*}%
Thus (\ref{SobFbInnProd}) reads%
\begin{equation}
\left\langle p,q\right\rangle _{s}=\left\langle p,q\right\rangle
+\sum_{k=0}^{s}M_{k}p^{(k)}(0)q^{(k)}(0).  \label{SobFbInnProd_diag}
\end{equation}

We will denote by $\{Q_{n}(x)\}_{n\geq 0}$ the MOPS with respect to the
above inner product. This is the so called \textit{diagonal case} for
Sobolev-type inner products, see \cite{AMRR-JAT15}. If there are no
derivatives involved therein (i.e., $s=0$), the polynomials orthogonal with
respect to (\ref{SobFbInnProd_diag})\ are known as \textit{Krall--type
orthogonal polynomials}, and they are orthogonal with respect to a standard
inner product, because the operator of multiplication by $x$ is symmetric
with respect to such an inner product, i.e. $\langle xp,q\rangle
_{s=0}=\langle p,xq\rangle _{s=0}$, for every $p,q\in \mathbb{P}$. On the
other hand, when $s>0$ (\ref{SobFbInnProd}) becomes non--standard, and the
corresponding polynomials are called \textit{Sobolev--type orthogonal
polynomials}. In this work we consider the Sobolev case, so we will refer $%
Q_{n}(x)$ as \textit{Freud--Sobolev type} orthogonal polynomials.

We will also use a notation relative to the norm of the polynomials. If for
any $n$-th degree polynomial of a sequence of orthogonal polynomials we have 
$\langle f_{n},f_{n}\rangle =||f_{n}||^{2}=1$, then the sequence is said to
be \textit{orthonormal}. In order to have uniqueness, we will always choose
the leading coefficient of any orthonormal polynomial to be positive for
every $n$.

\bigskip

\begin{proposition}
\label{[Sec1]-PROP1}Let $\{f_{n}(x)\}_{n\geq 0}$ denote the sequence of
polynomials orthonormal with respect to (\ref{Fb-InnerProduct}). That is,%
\begin{equation*}
f_{n}(x)=\gamma _{n}F_{n}(x)=\gamma _{n}x^{n}+\text{\textit{lower degree
terms}},
\end{equation*}%
where%
\begin{equation}
\gamma _{n}=\left( ||F_{n}||^{2}\right) ^{-1/2}>0,  \label{[Sec1]-leadcoefp}
\end{equation}%
and%
\begin{equation*}
||F_{n}||^{2}=\int_{\mathbb{R}}[F_{n}(x)]^{2}e^{-x^{4}}dx.
\end{equation*}%
The following structural properties hold.

\begin{enumerate}
\item Three term recurrence relation (see \cite{N-CMS83}). Since $\omega (x)$%
\ is an even weight function, the family $\{f_{n}(x)\}_{n\geq 0}$ is
symmetric. For every $n\in \mathbb{N}$,%
\begin{equation}
xf_{n-1}(x)=a_{n}f_{n}(x)+a_{n-1}f_{n-2}(x),\quad n\geq 1,  \label{3TRR}
\end{equation}

with $f_{-1}:=0$, $f_{0}(x)=(\int_{\mathbb{R}}\omega (x)dx)^{-1/2}$, $%
\,f_{1}(x)=a_{1}^{-1}x$. Also, $a_{n}=\frac{\gamma _{n-1}}{\gamma _{n}},
n\geq1$, $a_{0}=0$, and 
\begin{equation*}
a_{1}^{2}=\frac{\int_{\mathbb{R}}x^{2}\omega (x)dx}{\int_{\mathbb{R}}\omega
(x)dx}=\frac{\Gamma \left( \frac{3}{4}\right) }{\Gamma \left( \frac{1}{4}%
\right) }.
\end{equation*}%
We also have (see \cite{FVZ-JPA12})%
\begin{equation}
xF_{n}(x)=F_{n+1}(x)+a_{n}^{2}F_{n-1}(x),\quad n\geq 1,  \label{3TRR-Monic}
\end{equation}%
for the monic normalization.

\item Ratio of the leading coefficients (see \cite{LQ-JAT83}, \cite{N-SJMA84}%
)%
\begin{equation}
a_{n}^{2}=\left( \frac{n}{12}\right) ^{\frac{1}{2}}\left[ 1+\frac{1}{24n^{2}}%
+\mathcal{O}(n^{-4})\right]  \label{[Sec1]-LewQuarles}
\end{equation}

\item String equation (see \cite[(2.12)]{V-WS07}). An important feature of
these polynomials is that the recurrence coefficients $a_{n}$\ in the above
three term recurrence relation, satisfy the following nonlinear difference
equation%
\begin{equation*}
4a_{n}^{2}\left( a_{n+1}^{2}+a_{n}^{2}+a_{n-1}^{2}\right) =n,n\geq 1.
\end{equation*}%
This is known in the literature as the \textit{string equation} or \textit{%
Freud equation }(see \cite{FVZ-JPA12}, \cite[(3.2.20)]{Ism05}, among others)

\item (\cite[Th. 4]{N-CMS83}).The polynomials $f_{n}(x)$ defined by (\ref%
{3TRR}) constitute a generalized Appell sequence. More precisely,%
\begin{equation*}
\lbrack f_{n}]^{\prime }(x)=\frac{n}{a_{n}}%
f_{n-1}(x)+4a_{n}a_{n-1}a_{n-2}f_{n-3}(x),\quad n=1,2,\ldots
\end{equation*}

\item (\cite[Th. 5]{N-CMS83}). The polynomials $f_{n}(x)$ satisfy%
\begin{equation}
\lbrack f_{n}]^{\prime }(x)=-4xa_{n}^{2}f_{n}(x)+4a_{n}\phi
_{n}(x)f_{n-1}(x),  \label{fprimDer}
\end{equation}%
and%
\begin{eqnarray}
\lbrack f_{n}]^{\prime \prime }(x)
&=&[16a_{n}^{4}x^{2}-4a_{n}^{2}-16a_{n}^{2}\phi _{n}(x)\phi
_{n-1}(x)]f_{n}(x)  \label{fsegDer} \\
&&+[8a_{n}x+16a_{n}x^{3}\phi _{n}(x)]f_{n-1}(x),  \notag
\end{eqnarray}%
where (see \cite[eq. (14)]{N-CMS83})%
\begin{equation*}
\phi _{n}(x)=a_{n+1}^{2}+a_{n}^{2}+x^{2}.
\end{equation*}

\item Strong inner asymptotics (see \cite[Th. 1]{N-CMS83}, and \cite[eq. (8)]%
{N-SJMA84}). Let $\{f_{n}\}_{n\geq 0}$ be the orthonormal polynomials with
respect to the weight function for $\omega (x)=e^{-x^{4}}$. Then,%
\begin{equation*}
f_{n}\left( x\right) =Ae^{x^{4}/2}n^{-1/8}\times 
\end{equation*}%
\begin{equation*}
\sin \left\{ \left( \frac{64}{27}\right)
^{1/4}n^{3/4}x+12^{-1/4}n^{1/4}x^{3}-\frac{n-1}{2}\pi \right\} +o(n^{-1/8}),
\end{equation*}%
where $A=\sqrt[8]{12}/\sqrt{\pi }$, uniformly for $x$ in every compact
subset $\Delta \subset \mathbb{R}$.
\end{enumerate}
\end{proposition}



We are interested in the asymptotic properties of derivatives of the Freud
polynomials which will be useful in the sequel. 
From \cite[eqs. (16)-(17)]{N-CMS83}, and (\ref{fsegDer}), the following
Lemma follows



\begin{lemma}
\label{[Sec2]-LEMMA1}(see \cite[Th. 6]{N-CMS83}) There exists a constant $A=%
\sqrt[8]{12}/\sqrt{\pi }$ \ such that the following estimates hold%
\begin{eqnarray*}
f_{n}(0) &=&\left\{ 
\begin{array}{ll}
0 & \text{if }n\text{\ is odd} \\ 
(-1)^{n/2}n^{-1/8}(A+o(1)) & \text{if }n\text{\ is even}%
\end{array}%
\right. , \\
\lbrack f_{n}]^{\prime }(0) &=&\left\{ 
\begin{array}{ll}
(-1)^{(n-1)/2}\frac{\sqrt{8}}{\sqrt[4]{27}}n^{5/8}(A+o(1)) & \text{if }n%
\text{\ is odd} \\ 
0 & \text{if }n\text{\ is even}%
\end{array}%
\right. , \\
\lbrack f_{n}]^{\prime \prime }(0) &=&\left\{ 
\begin{array}{ll}
0 & \text{if }n\text{\ is odd} \\ 
\left( -1\right) ^{\frac{n}{2}+1}\frac{8}{3\sqrt{3}}n^{11/8}(A+o(1)) & \text{%
if }n\text{\ is even}%
\end{array}%
\right. , \\
\lbrack f_{n}]^{\prime \prime \prime }(0) &=&\left\{ 
\begin{array}{ll}
(-1)^{(n-1)/2}\frac{16\sqrt{2}}{9\sqrt[4]{3}}n^{17/8}(A+o(1)) & \text{if }n%
\text{\ is odd} \\ 
0 & \text{if }n\text{\ is even}%
\end{array}%
\right. .
\end{eqnarray*}
\end{lemma}



The kernel polynomials associated with the polynomial sequence $%
\{F_{n}\}_{n\geq 0}$ will play a key role in order to prove some of the
results of the manuscript. In the remaining of this section, we analyze them
in detail. The $n$-th degree reproducing kernel associated with $\left\{
F_{n}\right\} _{n\geqslant 0}$ is%
\begin{equation*}
K_{n}(x,y)=\sum\limits_{k=0}^{n}\frac{F_{k}(x)F_{k}(y)}{||F_{k}||^{2}}.
\end{equation*}%
For $x\neq y$, the Christoffel-Darboux formula reads%
\begin{equation}
K_{n}(x,y)=\frac{1}{||F_{n}||^{2}}\frac{F_{n+1}(x)F_{n}(y)-F_{n+1}(y)F_{n}(x)%
}{x-y},  \label{CristDar}
\end{equation}%
and its confluent expression becomes%
\begin{equation}
K_{n}(x,x)=\sum_{k=0}^{n}\frac{[F_{k}(x)]^{2}}{||F_{k}||^{2}}=\frac{%
[F_{n+1}]^{\prime }(x)F_{n}(x)-[F_{n}]^{\prime }(x)F_{n+1}(x)}{||F_{n}||^{2}}%
.  \label{[Sec1]-Knxx-P}
\end{equation}%
We introduce the following standard notation for the partial derivatives of
the $n$-th degree kernel $K_{n}(x,y)$%
\begin{equation}
\frac{\partial ^{j+k}K_{n}\left( x,y\right) }{\partial ^{j}x\partial ^{k}y}%
=:K_{n}^{(j,k)}\left( x,y\right) ,\quad 0\leq j,k\leq n.  \label{KDerivij}
\end{equation}%
Thus,%
\begin{equation}
K_{n-1}^{(0,1)}(x,0)=K_{n-1}^{(1,0)}(0,x)=\frac{1}{||F_{n-1}||^{2}}\cdot
\label{[Sec1]-Knxy01der-P}
\end{equation}%
\begin{equation}
\left[ \frac{F_{n}(x)F_{n-1}(0)-F_{n-1}(x)F_{n}(0)}{x^{2}}+\frac{%
F_{n}(x)[F_{n-1}]^{\prime }(0)-F_{n-1}(x)[F_{n}]^{\prime }(0)}{x}\right] , 
\notag
\end{equation}%
and, considering the coefficient of $x$ in the above expression, we have%
\begin{equation*}
K_{n-1}^{(1,1)}(0,0)=\frac{1}{||F_{n-1}||^{2}}\cdot
\end{equation*}%
\begin{equation*}
\left[ \frac{\lbrack F_{n}]^{\prime \prime \prime
}(0)F_{n-1}(0)-[F_{n-1}]^{\prime \prime \prime }(0)F_{n}(0)}{6}+\frac{%
[F_{n}]^{\prime \prime }(0)[F_{n-1}]^{\prime }(0)-[F_{n-1}]^{\prime \prime
}(0)[F_{n}]^{\prime }(0)}{2}\right] .
\end{equation*}%
From (\ref{[Sec1]-Knxx-P})%
\begin{equation}
K_{n-1}(0,0)=\frac{[F_{n}]^{\prime }(0)F_{n-1}(0)-[F_{n-1}]^{\prime
}(0)F_{n}(0)}{||F_{n-1}||^{2}},  \label{[Sec2]-K00cc}
\end{equation}%
and taking limit in (\ref{[Sec1]-Knxy01der-P}) when $x\rightarrow 0$, we get 
\begin{equation*}
K_{n-1}^{(0,1)}(0,0)=K_{n-1}^{(1,0)}(0,0)=\frac{1}{||F_{n-1}||^{2}}\frac{%
[F_{n}]^{\prime \prime }(0)F_{n-1}(0)-[F_{n-1}]^{\prime \prime }(0)F_{n}(0)}{%
2}.
\end{equation*}%
%
%
%
%
%
%
%
%
%
%
%
%
%
%
%
%
%
%
%
%
%
%
%
%
%
%
%
%
%
%
Taking a suitable index shifting in the last three expressions, we conclude%
\begin{eqnarray*}
K_{2n-1}(0,0) &=&\frac{-[F_{2n-1}]^{\prime }(0)F_{2n}(0)}{||F_{2n-1}||^{2}},
\\
K_{2n-1}^{(0,1)}(0,0) &=&K_{2n-1}^{(1,0)}(0,0)=0, \\
K_{2n-1}^{(1,1)}(0,0) &=&\frac{1}{||F_{2n-1}||^{2}}\left[ \frac{%
[F_{2n}]^{\prime \prime }(0)[F_{2n-1}]^{\prime }(0)}{2}-\frac{%
[F_{2n-1}]^{\prime \prime \prime }(0)F_{2n}(0)}{6}\right]
\end{eqnarray*}%
as well as 
\begin{eqnarray*}
K_{2n}(0,0) &=&\frac{[F_{2n+1}]^{\prime }(0)F_{2n}(0)}{||F_{2n}||^{2}}, \\
K_{2n}^{(0,1)}(0,0) &=&K_{2n}^{(1,0)}(0,0)=0, \\
K_{2n}^{(1,1)}(0,0) &=&\frac{1}{||F_{2n}||^{2}}\left[ \frac{%
[F_{2n+1}]^{\prime \prime \prime }(0)F_{2n}(0)}{6}-\frac{[F_{2n}]^{\prime
\prime }(0)[F_{2n+1}]^{\prime }(0)}{2}\right] .
\end{eqnarray*}



Another interesting property of the Freud kernels arises from the symmetry
of $\{F_{n}(x)\}_{n\geq 0}$. From (\ref{[Sec1]-Knxx-P}) and (\ref{KDerivij})
we have%
\begin{eqnarray*}
K_{2n+1}(x,0) &=&\sum_{i=0}^{n-1}{\frac{F_{2i}(0)}{\left\Vert {F_{2i}}%
\right\Vert ^{2}}}\,F_{2i}(x)=K_{2n}(x,0), \\
K_{2n}^{(0,1)}(x,0) &=&K_{2n-1}^{(0,1)}(x,0), \\
K_{2n}^{(1,1)}(x,0) &=&{\frac{[F_{2n}]^{\prime }(x)[F_{2n}]^{\prime }(0)}{%
\left\Vert {F_{2n}}\right\Vert ^{2}}}%
+K_{2n-1}^{(1,1)}(x,0)=K_{2n-1}^{(1,1)}(x,0),
\end{eqnarray*}%
This fact will be useful throughout the paper. We also need explicit
asymptotic expressions for the reproducing kernel and their derivatives.
They are introduced in the following lemma. 


\begin{lemma}
\label{[Sec2]-LEMMA2}For every $n=0,1,\ldots $, we have%
\begin{eqnarray*}
K_{n}(0,0) &=&\mathcal{O}(n^{3/4}), \\
K_{n}^{(0,1)}(0,0) &=&K_{n-1}^{(1,0)}(0,0)=0, \\
K_{n}^{(1,1)}(0,0) &=&\mathcal{O}(n^{9/4}).
\end{eqnarray*}
\end{lemma}



\begin{proof}
Writing $K_{n}(0,0)$, $K_{n}^{(0,1)}(0,0)$ and $K_{n}^{(1,1)}(0,0)$ in terms
of orthonormal polynomials $f_{n}$,%
the Lemma follows.
\end{proof}

The structure of the manuscript is as follows.

In Section 2 we will obtain connection formulas between monic Freud-Sobolev
type and monic Freud orthogonal polynomials. We also prove that
Freud-Sobolev orthogonal polynomials satisfy a five term recurrence relation
and we will deduce the asymptotic behavior of the coefficients involved
therein. In Section 3 we study some analytic properties of zeros of
Freud-Sobolev type orthogonal polynomials, in particular interlacing and
asymptotic behavior. Section 4 is focused on the second order linear
differential equation that such polynomials satisfy. As a direct consequence,
the electrostatic interpretation of such polynomials in terms of a
logarithmic potential interaction and an external potential is presented.



\section{Connection formulas}

\label{[Sec-2]-ConnForm}



Let us consider the aforementioned Sobolev-type inner product (\ref%
{SobFbInnProd_diag}). In the sequel, we will denote by $\{Q_{n}(x)\}_{n\geq
0}$\ the corresponding sequence of monic orthogonal polynomials and by%
\begin{equation*}
||Q_{n}||_{s}^{2}=\langle Q_{n},x^{n}\rangle _{s}
\end{equation*}%
the norm of the $n$-th degree polynomial. The connection formula between $%
\{Q_{n}(x)\}_{n\geq 0}$ and $\{P_{n}(x)\}_{n\geq 0}$ is stated in the
following lemma.

\begin{lemma}
\cite{AMRR-JAT15} For $n\geq 1$, we have 
\begin{equation}
Q_{n}(x)=F_{n}\left( x\right)
-\sum_{k=0}^{s}M_{k}[Q_{n}]^{(k)}(0)K_{n-1}^{(0,k)}(x,0),  \label{expan-2}
\end{equation}%
where, for $0\leq k\leq s$,%
\begin{equation*}
\lbrack Q_{n}]^{(k)}(0)=\left( \det D\right) ^{-1}%
\begin{vmatrix}
1+M_{0}K_{n-1}^{(0,0)}(0,0) & \cdots & F_{n}(0) & \cdots & 
M_{s}K_{n-1}^{(0,s)}(0,0) \\ 
M_{0}K_{n-1}^{(1,0)}(0,0) & \cdots & [F_{n}]^{\prime }(0) & \cdots & 
M_{s}K_{n-1}^{(1,s)}(0,0) \\ 
\vdots &  & \vdots & \ddots & \vdots \\ 
M_{0}K_{n-1}^{(s,0)}(0,0) & \cdots & [F_{n}]^{(s)}(0) & \cdots & 
1+M_{s}K_{n-1}^{(s,s)}(0,0)%
\end{vmatrix}%
,
\end{equation*}%
with%
\begin{equation*}
D=%
\begin{bmatrix}
1+M_{0}K_{n-1}^{(0,0)}(0,0) & M_{1}K_{n-1}^{(0,1)}(0,0) & \cdots & 
M_{s}K_{n-1}^{(0,s)}(0,0) \\ 
M_{0}K_{n-1}^{(1,0)}(0,0) & 1+M_{1}K_{n-1}^{(1,1)}(0,0) & \cdots & 
M_{s}K_{n-1}^{(1,s)}(0,0) \\ 
\vdots & \vdots & \ddots & \vdots \\ 
M_{0}K_{n-1}^{(s,0)}(0,0) & M_{1}K_{n-1}^{(s,1)}(0,0) & \cdots & 
1+M_{s}K_{n-1}^{(s,s)}(0,0)%
\end{bmatrix}%
.
\end{equation*}
\end{lemma}

Moreover, an easy computation shows that%
\begin{equation*}
K_{n-1}^{(0,k)}(x,0)=\frac{1}{||F_{n-1}||^{2}}\left( \sum_{\eta =0}^{k}\frac{%
k!}{\eta !}\frac{F_{n}(x)[F_{n-1}]^{(\eta )}(0)-F_{n-1}(x)[F_{n}]^{(\eta
)}(0)}{x^{k-\eta +1}}\right) ,
\end{equation*}%
and, as a consequence, we can write (\ref{expan-2}) as 
\begin{equation}
x^{s+1}Q_{n}(x)=\mathcal{A}_{s}(n;x)F_{n}(x)+\mathcal{B}_{s}(n;x)F_{n-1}(x),
\label{FormConex-3}
\end{equation}%
where%
\begin{eqnarray*}
\mathcal{A}_{s}(n;x) &=&\sum_{k=0}^{s}\left( x^{s+1}-\sum_{\eta =0}^{k}\frac{%
k!}{\eta !}\frac{M_{k}[Q_{n}]^{(k)}(0)[F_{n-1}]^{(\eta )}(0)}{||F_{n-1}||^{2}%
}x^{s-k+\eta }\right) , \\
\mathcal{B}_{s}(n;x) &=&\sum_{k=0}^{s}\left( \sum_{\eta =0}^{k}\frac{k!}{%
\eta !}\frac{M_{k}[Q_{n}]^{(k)}(0)[F_{n}]^{(\eta )}(0)}{||F_{n-1}||^{2}}%
x^{s-k+\eta }\right) ,
\end{eqnarray*}%
are polynomials of degree $s+1$\ and $s$, respectively.



\subsection{Connection formulas for monic polynomials}




In what follows, we restrict ourselves to study the case of only one mass
point with derivative in the inner product (\ref{Fb-InnerProduct}), i.e., $%
s=1$, $M_{0}\geq 0$, and $M_{1}\geq 0$,%
\begin{equation}
\left\langle p,q\right\rangle _{1}=\left\langle p,q\right\rangle
+M_{0}p(0)q(0)+M_{1}p^{\prime }(0)q^{\prime }(0).  \label{InnerDelta(0)}
\end{equation}%
In such a case, the connection formula (\ref{FormConex-3}) becomes%
\begin{equation}
x^{2}Q_{n}(x)=\mathcal{A}_{1}(n;x)F_{n}(x)+\mathcal{B}_{1}(n;x)F_{n-1}(x),
\label{CF-1}
\end{equation}%
where $\mathcal{A}_{1}(n;x)=x^{2}+\mathcal{A}_{10}(n)$, and $\mathcal{B}%
_{1}(n;x)=\mathcal{B}_{11}(n)x$ with%
\begin{equation*}
\mathcal{A}_{10}(n)=-\frac{M_{1}[Q_{n}]^{\prime }(0)F_{n-1}(0)}{%
||F_{n-1}||^{2}},\quad \mathcal{B}_{11}(n)=\frac{%
M_{0}Q_{n}(0)F_{n}(0)+M_{1}[Q_{n}]^{\prime }(0)[F_{n}]^{\prime }(0)}{%
||F_{n-1}||^{2}}.
\end{equation*}

To obtain $Q_{n}(0)$ and $[Q_{n}]^{\prime }(0)$ in the above expression, we
evaluate (\ref{CF-1}) at $x=0$ and solve the corresponding linear system.
Indeed,%
\begin{eqnarray*}
Q_{n}(0) &=&\frac{%
\begin{vmatrix}
F_{n}(0) & M_{1}K_{n-1}^{(0,1)}(0,0) \\ 
\lbrack F_{n}]^{\prime }(0) & 1+M_{1}K_{n-1}^{(1,1)}(0,0)%
\end{vmatrix}%
}{%
\begin{vmatrix}
1+M_{0}K_{n-1}(0,0) & M_{1}K_{n-1}^{(0,1)}(0,0) \\ 
M_{0}K_{n-1}^{(1,0)}(0,0) & 1+M_{1}K_{n-1}^{(1,1)}(0,0)%
\end{vmatrix}%
}, \\
\lbrack Q_{n}]^{\prime }(0) &=&\frac{%
\begin{vmatrix}
1+M_{0}K_{n-1}(0,0) & F_{n}(0) \\ 
M_{0}K_{n-1}^{(1,0)}(0,0) & [F_{n}]^{\prime }(0)%
\end{vmatrix}%
}{%
\begin{vmatrix}
1+M_{0}K_{n-1}(0,0) & M_{1}K_{n-1}^{(0,1)}(0,0) \\ 
M_{0}K_{n-1}^{(1,0)}(0,0) & 1+M_{1}K_{n-1}^{(1,1)}(0,0)%
\end{vmatrix}%
}.
\end{eqnarray*}%
As a consequence, 
\begin{eqnarray}
Q_{n}(0) &=&\displaystyle\frac{F_{n}(0)}{[1+M_{0}K_{n-1}(0,0)]},
\label{QenF1} \\
\lbrack Q_{n}]^{\prime }(0) &=&{\displaystyle\frac{[F_{n}]^{\prime }(0)}{%
1+M_{1}K_{n-1}^{(1,1)}(0,0)}}.  \label{QenF2}
\end{eqnarray}%
Thus, 
\begin{equation}
\begin{array}{rcllrl}
Q_{2n}(0) &  & ={\displaystyle\frac{F_{2n}(0)}{[1+M_{0}K_{2n-2}(0,0)]}}, & 
& Q_{2n+1}(0) & =0, \\ 
\lbrack Q_{2n+1}]^{\prime }(0) &  & ={\displaystyle\frac{[F_{2n+1}]^{\prime
}(0)}{1+M_{1}K_{2n-1}^{(1,1)}(0,0)}}, &  & [Q_{2n}]^{\prime }(0) & =0.%
\end{array}
\label{[Sec2]-Qn(0)Monic}
\end{equation}

Let us denote by $\{q_{n}(x)\}_{n\geq 0}$, with 
\begin{equation*}
q_{n}(x)=\zeta _{n}x^{n}+\text{\textit{\ lower degree terms,}}
\end{equation*}%
the sequence of Freud-Sobolev type orthonormal polynomials. The relation
between the leading coefficients $\zeta _{n}$ and $\gamma _{n}$ is given in
the following result.



\begin{proposition}
For $n\geq 1$, we have 
\begin{equation}
\begin{array}{r}
\displaystyle\frac{\Vert F_{2n}\Vert ^{2}}{\Vert Q_{2n}\Vert ^{2}}={%
\displaystyle\frac{1+M_{0}K_{2n-2}(0,0)}{1+M_{0}K_{2n}(0,0)}}, \\ 
\\ 
\displaystyle\frac{\Vert F_{2n+1}\Vert ^{2}}{\Vert Q_{2n+1}\Vert ^{2}}={%
\displaystyle\frac{1+M_{1}K_{2n-1}^{(1,1)}(0,0)}{1+M_{1}K_{2n+1}^{(1,1)}(0,0)%
}}.%
\end{array}
\label{Lema2}
\end{equation}
\end{proposition}



\begin{proof}
Consider the Fourier expansion%
\begin{equation*}
q_{n}(x)=\sum_{k=0}^{n}c_{k,n}\,f_{k}(x),
\end{equation*}%
whose coefficients are (see (\ref{InnerDelta(0)}))%
\begin{equation*}
c_{k,n}=\int_{\mathbb{R}}q_{n}(x)f_{k}(x)e^{-x^{4}}dx=\langle
q_{n}(x),f_{k}(x)\rangle _{1}-M_{0}q_{n}(0)f_{k}(0)-M_{1}[q_{n}]^{\prime
}(0)[f_{k}]^{\prime }(0).
\end{equation*}%
If $k=n$, by comparing the leading coefficients, we obtain $c_{n,n}=\zeta
_{n}/\gamma _{n}$. When $k<n$, by orthogonality we have $\langle
q_{n}(x),f_{k}(x)\rangle _{1}=0$, so that%
\begin{equation*}
c_{k,n}=-M_{0}q_{n}(0)f_{k}(0)-M_{1}[q_{n}]^{\prime }(0)[f_{k}]^{\prime }(0).
\end{equation*}%
Hence,%
\begin{eqnarray*}
\int_{\mathbb{R}}[q_{n}(x)]^{2}e^{-x^{4}}dx &=&\left( \frac{\zeta _{n}}{%
\gamma _{n}}\right) ^{2}+\sum_{k=0}^{n-1}[c_{k,n}]^{2}[f_{k}(x)]^{2} \\
&=&\left( \frac{\zeta _{n}}{\gamma _{n}}\right)
^{2}+M_{0}^{2}q_{n}^{2}(0)\sum_{k=0}^{n-1}f_{k}^{2}(0)+M_{1}^{2}([q_{n}]^{%
\prime }(0))^{2}\sum_{k=0}^{n-1}([f_{k}]^{\prime }(0))^{2} \\
&&+2M_{0}M_{1}q_{n}(0)[q_{n}]^{\prime
}(0)\sum_{k=0}^{n-1}f_{k}(0)[f_{k}]^{\prime }(0).
\end{eqnarray*}%
On the other hand, by the orthonormality of $q_{n}(x)$ with respect to (\ref%
{InnerDelta(0)}),%
\begin{equation*}
\left\langle q_{n},q_{n}\right\rangle _{1}=1=\int_{\mathbb{R}%
}[q_{n}(x)]^{2}e^{-x^{4}}dx+M_{0}q_{n}^{2}(0)+M_{1}([q_{n}]^{\prime
}(0))^{2},
\end{equation*}%
so that%
\begin{equation*}
\int_{\mathbb{R}%
}[q_{n}(x)]^{2}e^{-x^{4}}dx=1-M_{0}q_{n}^{2}(0)-M_{1}([q_{n}]^{\prime
}(0))^{2}.
\end{equation*}%
Therefore,%
\begin{equation*}
1-M_{0}q_{n}^{2}(0)-M_{1}([q_{n}]^{\prime }(0))^{2}=\left( \frac{\zeta _{n}}{%
\gamma _{n}}\right) ^{2}+M_{0}^{2}q_{n}^{2}(0)\sum_{k=0}^{n-1}f_{k}^{2}(0)
\end{equation*}%
\begin{equation*}
+M_{1}^{2}([q_{n}]^{\prime }(0))^{2}\sum_{k=0}^{n-1}([f_{k}]^{\prime
}(0))^{2}+2M_{0}M_{1}q_{n}(0)[q_{n}]^{\prime
}(0)\sum_{k=0}^{n-1}f_{k}(0)[f_{k}]^{\prime }(0).
\end{equation*}%
Taking into account that $\sum_{k=0}^{n-1}f_{k}(0)[f_{k}]^{\prime
}(0)=K_{n-1}^{(0,1)}(0,0)=0$, we rewrite the above expression as%
\begin{equation*}
1-M_{0}q_{n}^{2}(0)-M_{1}([q_{n}]^{\prime }(0))^{2}=\left( \frac{\zeta _{n}}{%
\gamma _{n}}\right)
^{2}+M_{0}^{2}q_{n}^{2}(0)K_{n-1}(0,0)+M_{1}^{2}([q_{n}]^{\prime
}(0))^{2}K_{n-1}^{(1,1)}(0,0),
\end{equation*}%
and, as a consequence,%
\begin{equation*}
1=\left( \frac{\zeta _{n}}{\gamma _{n}}\right) ^{2}+M_{0}q_{n}^{2}(0)\left[
1+M_{0}K_{n-1}(0,0)\right] +M_{1}([q_{n}]^{\prime }(0))^{2}\left[
1+M_{1}K_{n-1}^{(1,1)}(0,0)\right] .
\end{equation*}%
Next, using the orthonormal versions of \eqref{QenF1} and \eqref{QenF2},
respectively, 
\begin{equation*}
q_{n}(0)=\frac{\zeta _{n}}{\gamma _{n}}\frac{f_{n}(0)}{1+M_{0}K_{n-1}(0,0)}%
,\quad \lbrack q_{n}]^{\prime }(0)=\frac{\zeta _{n}}{\gamma _{n}}\frac{%
[f_{n}]^{\prime }\left( 0\right) }{1+M_{1}K_{n-1}^{(1,1)}(0,0)},
\end{equation*}%
we obtain%
\begin{equation*}
1=\left( \frac{\zeta _{n}}{\gamma _{n}}\right) ^{2}\left[ 1+M_{0}\frac{%
(f_{n}(0))^{2}}{\left[ 1+M_{0}K_{n-1}(0,0)\right] }+M_{1}\frac{%
([f_{n}]^{\prime }\left( 0\right) )^{2}}{\left[ 1+M_{1}K_{n-1}^{(1,1)}(0,0)%
\right] }\right] .
\end{equation*}%
Since%
\begin{eqnarray*}
K_{n}(0,0)-K_{n-1}(0,0) &=&(f_{n}(0))^{2} \\
K_{n}^{(1,1)}(0,0)-K_{n-1}^{(1,1)}(0,0) &=&([f_{n}]^{\prime }(0))^{2}
\end{eqnarray*}%
we get%
\begin{eqnarray*}
1 &=&\left( \frac{\zeta _{n}}{\gamma _{n}}\right) ^{2}\left[ 1+\frac{%
M_{0}K_{n}(0,0)-M_{0}K_{n-1}(0,0)}{\left[ 1+M_{0}K_{n-1}(0,0)\right] }+\frac{%
M_{1}K_{n}^{(1,1)}(0,0)-M_{1}K_{n-1}^{(1,1)}(0,0)}{%
1+M_{1}K_{n-1}^{(1,1)}(0,0)}\right] , \\
1 &=&\left( \frac{\zeta _{n}}{\gamma _{n}}\right) ^{2}\left[ \frac{%
1+M_{0}K_{n}(0,0)}{1+M_{0}K_{n-1}(0,0)}+\frac{1+M_{1}K_{n}^{(1,1)}(0,0)}{%
1+M_{1}K_{n-1}^{(1,1)}(0,0)}-1\right]
\end{eqnarray*}%
which is (\ref{Lema2}).

\end{proof}

Now, we obtain connection formulas that relate both families of monic
orthogonal polynomials.


%
%

\begin{proposition}
\label{[Sec2]-PROP21} The Freud-Sobolev type orthogonal polynomials satisfy 
\begin{equation}  \label{conexiongeneral}
x^2Q_{n}(x)=\left[ x^2-\frac{r_{n}\kappa _{n}^{[1]}}{4\phi _{n}(0)}\right]
F_{n}(x)+a_{n}^{2}\left( \kappa _{n}^{[0]}{+}\kappa _{n}^{[1]}\right)
xF_{n-1}(x), \quad n\geq 1,
\end{equation}
where 
\begin{eqnarray*}
\kappa _{n}^{[0]} &=&\frac{1+M_{0}K_{n}(0,0)}{1+M_{0}K_{n-1}(0,0)}-1,\quad
\kappa _{n}^{[1]}=\frac{1+M_{1}K_{n}^{(1,1)}(0,0)}{%
1+M_{1}K_{n-1}^{(1,1)}(0,0)}-1, \quad r_{n} =\frac{1-(-1)^{n}}{2}.
\end{eqnarray*}%
Moreover, for the even and odd degrees, respectively, we have 
\begin{eqnarray}
Q_{2n}(x)&=&F_{2n}\left( x\right) -M_{0}\frac{F_{2n}(0)}{%
[1+M_{0}K_{2n-2}(0,0)]}K_{2n-2}(x,0), \quad n\geq 1,
\label{[Sec2]-Prop21iKer} \\
Q_{2n+1}(x)&=&F_{2n+1}\left( x\right) -M_{1}\frac{[F_{2n+1}]^{\prime }(0)}{%
1+M_{1}K_{2n-1}^{(1,1)}(0,0)}K_{2n-1}^{(0,1)}(x,0), \quad n\geq 1.
\label{[Sec2]-Prop21iiKer}
\end{eqnarray}
In other words, $Q_{2n}$ ($Q_{2n+1}$) is an even (odd) polynomial. 
%
\end{proposition}



\begin{proof}
Setting $s=1$ in (\ref{expan-2}) we get%
\begin{equation*}
Q_{n}(x)=F_{n}\left( x\right)
-M_{0}Q_{n}(0)K_{n-1}(x,0)-M_{1}[Q_{n}]^{\prime }(0)K_{n-1}^{(0,1)}(x,0).
\end{equation*}%
From \eqref{[Sec1]-Knxy01der-P} we have 
\begin{equation*}
K_{n-1}^{(0,1)}(x,0)=\left( \frac{F_{n-1}(0)+x[F_{n-1}]^{\prime }(0)}{%
x^{2}||F_{n-1}||^{2}}\right) F_{n}(x)-\left( \frac{F_{n}(0)+x[F_{n}]^{\prime
}(0)}{x^{2}||F_{n-1}||^{2}}\right) F_{n-1}(x),
\end{equation*}%
and taking into account \eqref{CristDar}, \eqref{QenF1}, \eqref{QenF2}, and
the symmetry of the Freud polynomials, we get

\begin{eqnarray*}
Q_{n}(x) &=&\left[ 1-\frac{F_{n-1}(0)}{[F_{n}]^{\prime }(0)}\frac{%
||F_{n}||^{2}}{x^{2}||F_{n-1}||^{2}}\left( \frac{1+M_{1}K_{n}^{(1,1)}(0,0)}{%
1+M_{1}K_{n-1}^{(1,1)}(0,0)}-1\right) \right] F_{n}(x) \\
&+&\frac{a_{n}^{2}}{x}\left[ \left( \frac{1+M_{0}K_{n}(0,0)}{%
1+M_{0}K_{n-1}(0,0)}-1\right) {+}\left( \frac{1+M_{1}K_{n}^{(1,1)}(0,0)}{%
1+M_{1}K_{n-1}^{(1,1)}(0,0)}-1\right) \right] F_{n-1}(x).
\end{eqnarray*}%
Denoting 
\begin{eqnarray*}
\kappa _{n}^{[0]} &=&\frac{1+M_{0}K_{n}(0,0)}{1+M_{0}K_{n-1}(0,0)}-1,\quad
\kappa _{n}^{[1]}=\frac{1+M_{1}K_{n}^{(1,1)}(0,0)}{%
1+M_{1}K_{n-1}^{(1,1)}(0,0)}-1, \quad r_{n} =\frac{1-(-1)^{n}}{2},
\end{eqnarray*}%
and noticing that from \eqref{fprimDer} we have $F_{n-1}(0)/[F_{n}]^{\prime
}(0)= 1/4a_{n}^{2}\phi _{n}(0)$, we obtain 
\begin{equation}  \label{conn_holo}
Q_{n}(x)=\left[ 1-\frac{r_{n}}{4x^{2}\phi _{n}(0)}\kappa _{n}^{[1]}\right]
F_{n}(x)+\frac{a_{n}^{2}}{x}\left( \kappa _{n}^{[0]}{+}\kappa
_{n}^{[1]}\right) F_{n-1}(x),
\end{equation}
which is \eqref{conexiongeneral}. On the other hand, shifting the index $%
n\rightarrow 2n$, and taking into account (\ref{[Sec2]-Qn(0)Monic}) we
obtain (\ref{[Sec2]-Prop21iKer}). For the odd case, (\ref{[Sec2]-Prop21iiKer}%
) follows similarly by using (\ref{[Sec2]-Qn(0)Monic}). 
%
%
\end{proof}

\begin{remark}
Notice that, from the symmetry of the Freud polynomials, we have $\kappa
_{2n+1}^{[0]}=0$ and $\kappa _{2n}^{[1]}=0$ for $n\geq 1$. As a consequence, %
\eqref{conexiongeneral} becomes 
\begin{eqnarray*}
xQ_{2n}(x) &=&xF_{2n}(x)+a_{2n}^{2}\kappa _{2n}^{[0]}F_{2n-1}(x),\quad n\geq
1, \\
x^{2}Q_{2n+1}(x) &=&\left[ x^{2}-\frac{r_{2n+1}\kappa _{2n+1}^{[1]}}{4\phi
_{2n+1}(0)}\right] F_{2n+1}(x)+a_{2n+1}^{2}\kappa
_{2n+1}^{[1]}xF_{2n}(x),\quad n\geq 1.
\end{eqnarray*}%
%
%
%
%
%
%
%
%
%
%
%
%
%
%
%
%
%
%
%
%
%
%
%
%
%
%
%
%
%
%
%
%
%
%
%
%
%
%
%
%
%
%
%
%
%
\end{remark}

The following is a straightforward extension of connection formulas %
\eqref{[Sec2]-Prop21iKer} and \eqref{[Sec2]-Prop21iiKer} for orthonormal
polynomials.

\begin{corollary}
\label{[Sec2]-PROP1}Let $q_{n}=\zeta _{n}x^{n}+\cdots $ and $f_{n}=\gamma
_{n}x^{n}+\cdots $ , then for $n\geq 1,$%
\begin{eqnarray*}
q_{2n}(x) &=&\frac{\zeta _{2n}}{\gamma _{2n}}%
f_{2n}(x)-M_{0}q_{2n}(0)K_{2n-2}(x,0), \\
q_{2n+1}(x) &=&\frac{\zeta _{2n+1}}{\gamma _{2n+1}}%
f_{2n+1}(x)-M_{1}[q_{2n+1}]^{\prime }(0)K_{2n-1}^{(0,1)}(x,0).
\end{eqnarray*}%
with%
\begin{equation*}
q_{2n}(0)=\frac{\zeta _{2n}}{\gamma _{2n}}\frac{f_{2n}(0)}{\left[
1+M_{0}K_{2n-2}(0,0)\right] },\quad \lbrack q_{2n+1}]^{\prime }(0)=\frac{%
\zeta _{2n+1}}{\gamma _{2n+1}}\frac{[f_{2n+1}]^{\prime }(0)}{\left[
1+M_{1}K_{2n-1}^{(1,1)}(0,0)\right] }.
\end{equation*}
\end{corollary}

\begin{remark}
Notice that, by defining $Q_{2n}(x):=P_n(x^2)$ and $Q_{2n+1}(x):=xR_n(x^2)$,
for $n\geq0$ and introducing the change of variable $x=\sqrt{y}$, we obtain
the following orthogonality relations 
\begin{eqnarray*}
0= \langle Q_{2n}, Q_{2m}\rangle=\int_{0}^\infty P_n (y)P_m(y)y^{-1/
2}e^{-y^2}dy+M_0P_n(0)P_m(0),\quad n\neq m \\
0=\langle Q_{2n+1}, Q_{2m+1}\rangle=\int_{0}^\infty R_n (y)R_m(y)y^{1/
2}e^{-y^2}dy+M_1R_n(0)R_m(0),\quad n\neq m,
\end{eqnarray*}
i.e. $\{P_{n} (x)\}_{n\geq0}$ and $\{R_{n}(x)\}_{n\geq0}$ are MOPS with
respect to standard inner products associated with the measures $%
d\sigma(x)=x^{-1/2}e^{-x^2}dx + M_0\delta(x)$ and $xd\sigma(x)+ M_1\delta(x)$%
, respectively, supported on the positive real semiaxis.
\end{remark}



\subsection{The five-term recurrence relation}

\label{[Sec-4]-5TRR}



This section deals with the five-term recurrence relation that the sequence
of discrete Freud--Sobolev orthogonal polynomials $\{Q_{n}(x)\}_{n\geq 0}$
satisfies. We will use the remarkable fact, which is a straightforward
consequence of (\ref{InnerDelta(0)}), that the multiplication operator by $%
x^{2}$ is a symmetric operator with respect to such a discrete Sobolev inner
product. Indeed, for polynomials $h(x),g(x)\in \mathbb{P}$ 
\begin{equation*}
\langle x^{2}h(x),g(x)\rangle _{1}=\langle h(x),x^{2}g(x)\rangle _{1}.
\end{equation*}

Notice that%
\begin{equation}
\langle x^{2}h(x),g(x)\rangle _{1}=\langle h(x),g(x)\rangle _{\lbrack 2]}.
\label{[Sec4]-Property1}
\end{equation}%
An equivalent formulation of (\ref{[Sec4]-Property1}) is%
\begin{equation*}
\langle x^{2}h(x),g(x)\rangle _{1}=\langle x^{2}h(x),g(x)\rangle .
\end{equation*}

We need a preliminary result.

\begin{lemma}
\label{[Sec4]-LEMMA-41}For every $n\geq 1$, the five term connection formula 
\begin{eqnarray*}
x^{2}Q_{n}(x) &=&F_{n+2}(x) +\left[ a_{n+1}^{2}+a_{n}^{2}+\mathcal{A}%
_{10}(n)+\mathcal{B}_{11}(n)\right] F_{n}(x) \\
&&+a_{n-1}^{2}\left[ a_{n}^{2}+\mathcal{B}_{11}(n)\right] F_{n-2}(x)
\end{eqnarray*}
holds.
\end{lemma}



\begin{proof}
The result follows easily from (\ref{CF-1}) after successive applications of
(\ref{3TRR-Monic}).
\end{proof}



We are ready to find the five-term recurrence relation satisfied by $%
\{Q_{n}(x)\}_{n\geq 0}$. Let us consider the Fourier expansion of $%
x^{2}Q_{n}(x)$ in terms of $\{Q_{n}(x)\}_{n\geq 0}$%
\begin{equation*}
x^{2}Q_{n}(x)=\sum_{k=0}^{n+2}\lambda _{n,k}Q_{k}(x),
\end{equation*}%
where%
\begin{equation}
\lambda _{n,k}=\frac{\langle x^{2}Q_{n}(x),Q_{k}(x)\rangle _{1}}{%
||Q_{k}||_{1}^{2}},\quad k=0,\ldots ,n+2.  \label{[Sec4]-CoefsS1}
\end{equation}%
Thus, $\lambda _{n,k}=0$ for $k=0,\ldots ,n-3$, and $\lambda _{n,n+2}=1$. To
obtain $\lambda _{n,n+1}$, we use (\ref{CF-1}) and get 
\begin{eqnarray*}
\lambda _{n,n+1} &=&\frac{1}{||Q_{n+1}||_{1}^{2}}\langle \mathcal{A}%
_{1}(n;x)F_{n}^{\alpha }(x),Q_{n+1}(x)\rangle _{1}+\frac{1}{%
||Q_{n+1}||_{1}^{2}}\langle \mathcal{B}_{1}(n;x)F_{n-1}(x),Q_{n+1}(x)\rangle
_{1} \\
&=&\frac{1}{||Q_{n+1}||_{1}^{2}}\langle x^{2}F_{n}(x),Q_{n+1}(x)\rangle _{1}=%
\frac{1}{||Q_{n+1}||_{1}^{2}}\langle F_{n}(x),x^{2}Q_{n+1}(x)\rangle =0,
\end{eqnarray*}%
by using Lemma \ref{[Sec4]-LEMMA-41}. In order to compute $\lambda _{n,n}$,
using (\ref{CF-1}) we get%
\begin{equation*}
\lambda _{n,n}=\frac{\langle x^{2}F_{n}(x),Q_{n}(x)\rangle _{1}}{%
||Q_{n}||_{1}^{2}}+\mathcal{A}_{10}(n)+\mathcal{B}_{11}(n).
\end{equation*}%
But, according to Lemma \ref{[Sec4]-LEMMA-41}, the first term is 
\begin{equation*}
\frac{\langle x^{2}F_{n}(x),Q_{n}(x)\rangle _{1}}{||Q_{n}||_{1}^{2}}=\left[
a_{n+1}^{2}+a_{n}^{2}+\mathcal{A}_{10}(n)+\mathcal{B}_{11}(n)\right] \frac{%
||F_{n}||^{2}}{||Q_{n}||_{1}^{2}},
\end{equation*}%
so that 
\begin{equation*}
\lambda _{n,n}=\left[ a_{n+1}^{2}+a_{n}^{2}+\mathcal{A}_{10}(n)+\mathcal{B}%
_{11}(n)\right] \frac{||F_{n}||^{2}}{||Q_{n}||_{1}^{2}}+\mathcal{A}_{10}(n)+%
\mathcal{B}_{11}(n).
\end{equation*}%
A similar analysis yields $\lambda _{n,n-1}=0$ and%
\begin{eqnarray*}
\lambda _{n,n-2} &=&\frac{\langle x^{2}Q_{n}(x),Q_{n-2}(x)\rangle }{%
||Q_{n-2}||_{1}^{2}}=\frac{\langle a_{n-1}^{2}\left[ a_{n}^{2}+\mathcal{B}%
_{11}(n)\right] F_{n-2}(x),Q_{n-2}(x)\rangle }{||Q_{n-2}||_{1}^{2}} \\
&=&a_{n-1}^{2}\left[ a_{n}^{2}+\mathcal{B}_{11}(n)\right] \frac{%
||F_{n-2}||^{2}}{||Q_{n-2}||_{1}^{2}}.
\end{eqnarray*}%
Thus, as a conclusion:

\begin{theorem}[Five-term recurrence relation]
For every $n\geq 1$, the monic Freud-Sobolev type polynomials $%
\{Q_{n}(x)\}_{n\geq 0}$, orthogonal with respect to (\ref{InnerDelta(0)}),
satisfy the following five-term recurrence relation%
\begin{equation}
x^{2}Q_{n}(x)=Q_{n+2}(x)+\lambda _{n,n}Q_{n}(x)+\lambda
_{n,n-2}Q_{n-2}(x),\quad n\geq 1,  \label{five_term_Q}
\end{equation}%
with initial conditions $Q_{-1}(x)=0,$ $Q_{0}(x)=1$, $Q_{1}(x)=x$, and $%
Q_{2}(x)=x^{2}-\lambda _{0,0}$, where%
\begin{eqnarray*}
\lambda _{n,n} &=&\left[ a_{n+1}^{2}+a_{n}^{2}+\mathcal{A}_{10}(n)+\mathcal{B%
}_{11}(n)\right] \frac{||F_{n}||^{2}}{||Q_{n}||_{1}^{2}}+\mathcal{A}_{10}(n)+%
\mathcal{B}_{11}(n),\quad n\geq 0, \\
\lambda _{n,n-2} &=&a_{n-1}^{2}\left[ a_{n}^{2}+\mathcal{B}_{11}(n)\right] 
\frac{||F_{n-2}||^{2}}{||Q_{n-2}||_{1}^{2}},\quad \quad n\geq 2.
\end{eqnarray*}
\end{theorem}

We now proceed to analyze the asymptotic behavior of the coefficients.
First, we need the following lemma.

\begin{lemma}
We have 
\begin{equation}  \label{normas}
\lim_{n\rightarrow\infty}\frac{\|Q_n\|}{\|F_n\|}=1+\mathcal{O}(n^{-1}).
\end{equation}
\end{lemma}

\begin{proof}
Let us consider first the even case. From \eqref{[Sec1]-leadcoefp} and its
analogue for $Q_{n}$, as well as \eqref{Lema2}, we have%
\begin{equation*}
\frac{\Vert Q_{2n}\Vert ^{2}}{\Vert F_{2n}\Vert ^{2}}={\displaystyle\frac{%
1+M_{0}K_{2n}(0,0)}{1+M_{0}K_{2n-2}(0,0)}=\frac{%
1+M_{0}(K_{2n-2}(0,0)+f_{2n}^{2}(0))}{1+M_{0}K_{2n-2}(0,0)}}.
\end{equation*}%
Taking into account Lemmas \ref{[Sec2]-LEMMA1} and \ref{[Sec2]-LEMMA2}, the
result follows. The odd case is similar.
\end{proof}


Notice that from successive applications of the three term recurrence
relation (\ref{3TRR-Monic}), we get 
\begin{equation*}
x^{2}F_{n}(x) =F_{n+2}(x) +[a_{n+1}^{2}+a_{n}^{2}] F_{n}(x) +a_{n-1}^{2}
a_{n}^{2} F_{n-2}(x).
\end{equation*}
We will show that, when $n\rightarrow\infty$, the five term recurrence
relation \eqref{five_term_Q} behaves exactly as the previous equation.

\begin{proposition}
We have%
\begin{equation*}
\lim_{n\rightarrow \infty }\frac{\lambda _{n,n}}{a_{n+1}^{2}+a_{n}^{2}}=1+%
\mathcal{O}(n^{-2}),\quad \text{and}\quad \lim_{n\rightarrow \infty }\frac{%
\lambda _{n,n-2}}{a_{n-1}^{2}a_{n}^{2}}=1+\mathcal{O}(n^{-3/2}).
\end{equation*}
\end{proposition}

\begin{proof}
In view of \eqref{five_term_Q} and \eqref{normas}, we need estimates for $%
\lim_{n\rightarrow \infty }\mathcal{A}_{10}(n)$ and $\lim_{n\rightarrow
\infty }\mathcal{B}_{11}(n)$. It is easy to show that $\mathcal{A}%
_{10}(2n)=0 $, and for the odd case we have%
\begin{equation*}
\mathcal{A}_{10}(2n+1)=-\frac{M_{1}[Q_{2n+1}]^{\prime }(0)F_{2n}(0)}{%
||F_{2n}||^{2}}=-\frac{M_{1}\left( \frac{[F_{2n+1}]^{\prime }(0)}{%
1+M_{1}K_{2n-1}^{(1,1)}(0,0)}\right) F_{2n}(0)}{||F_{2n}||^{2}}=-\frac{%
M_{1}K_{2n}(0,0)}{1+M_{1}K_{2n-1}^{(1,1)}(0,0)},
\end{equation*}%
where the second equality follows from \eqref{[Sec2]-Qn(0)Monic} and the
third equality from the confluent expression \eqref{[Sec2]-K00cc}. As a
consequence, using Lemma \ref{[Sec2]-LEMMA2}, we get $\mathcal{A}_{10}(2n+1)=%
\mathcal{O}(n^{-3/2})$. On the other side, for the even case, 
\begin{equation*}
\mathcal{B}_{11}(2n)=\frac{M_{0}Q_{2n}(0)F_{2n}(0)}{||F_{2n-1}||^{2}}=\frac{%
M_{0}F_{2n}^{2}(0)}{||F_{2n-1}||^{2}(1+M_{0}K_{2n-2}(0,0))}=\frac{M_{0}\Vert
F_{2n}\Vert ^{2}f_{2n}^{2}(0)}{\Vert F_{2n-1}\Vert ^{2}(1+M_{0}K_{2n-2}(0,0))%
},
\end{equation*}%
where we have used \eqref{[Sec2]-Qn(0)Monic} for the second equality and the
fact that $F_{2n}^{2}(0)=\Vert F_{2n}\Vert
^{2}(K_{2n}(0,0)-K_{2n-1}(0,0))=\Vert F_{2n}\Vert ^{2}f_{2n}^{2}(0)$ for the
third equality. Since $\Vert F_{2n}\Vert ^{2}/\Vert F_{2n-1}\Vert
^{2}=a_{2n}^{2}$, and by using \ref{[Sec1]-LewQuarles}, and Lemmas \ref%
{[Sec2]-LEMMA1} and \ref{[Sec2]-LEMMA2}, we obtain $\mathcal{B}_{11}(2n)=%
\mathcal{O}(n^{-1/2})$. Finally, for the odd case, in a similar way we have 
\begin{equation*}
\mathcal{B}_{11}(2n+1)=\frac{M_{1}([F_{2n+1}]^{\prime }(0))^{2}}{%
||F_{2n}||^{2}(1+M_{1}K_{2n-1}^{(1,1)}(0,0))}=\frac{M_{1}\Vert F_{2n+1}\Vert
^{2}(f_{2n+1}^{\prime }(0))^{2}}{\Vert F_{2n}\Vert
^{2}(1+M_{1}K_{2n-1}^{(1,1)}(0,0))},
\end{equation*}%
and again from \eqref{[Sec1]-LewQuarles}, and Lemmas \ref{[Sec2]-LEMMA1} and %
\ref{[Sec2]-LEMMA2}, we get $\mathcal{B}_{11}(2n+1)=\mathcal{O}(n^{-1/2})$.
As a consequence, we have%
\begin{equation*}
\lim_{n\rightarrow \infty }\frac{\lambda _{n,n}}{a_{n+1}^{2}+a_{n}^{2}}%
=\lim_{n\rightarrow \infty }\frac{\left[ a_{n+1}^{2}+a_{n}^{2}+\mathcal{A}%
_{10}(n)+\mathcal{B}_{11}(n)\right] \frac{||F_{n}||^{2}}{||Q_{n}||_{1}^{2}}+%
\mathcal{A}_{10}(n)+\mathcal{B}_{11}(n),}{a_{n+1}^{2}+a_{n}^{2}}=1+\mathcal{O%
}(n^{-2}),
\end{equation*}%
and%
\begin{equation*}
\lim_{n\rightarrow \infty }\frac{\lambda _{n,n-2}}{a_{n-1}^{2}a_{n}^{2}}=%
\frac{a_{n-1}^{2}\left[ a_{n}^{2}+\mathcal{B}_{11}(n)\right] \frac{%
||F_{n-2}||^{2}}{||Q_{n-2}||_{1}^{2}}.}{a_{n-1}^{2}a_{n}^{2}}=1+\mathcal{O}%
(n^{-3/2}).
\end{equation*}
\end{proof}



\section{The Zeros}

\label{[Sec-6]-Zeros}



In this Section we analyze some properties of the zeros of the polynomials $%
\{Q_{n}(x)\}_{n\geq 0}$.

\bigskip



\subsection{Interlacing rupture}



From (\ref{[Sec2]-Prop21iKer}) and (\ref{[Sec2]-Prop21iiKer}), it is clear
that the zeros of even $Q_{2n}(x)$ and odd $Q_{2n+1}(x)$ Freud-Sobolev type
polynomials act in an independent way. From those expressions, we observe
that the variation of $M_{0}$\ (respectively $M_{1}$) exclusively influences
the position of the zeros of $Q_{2n}(x)$\ (respectively $Q_{2n+1}(x)$)
without affecting the zeros of $Q_{2n+1}(x)$\ (respectively $Q_{2n}(x)$).
This interesting phenomena leads to the destruction of the zero interlacing
for two consecutive polynomials of the sequence $\{Q_{n}(x)\}_{n\geq 0}$ for
certain values of $M_{0}$ and $M_{1}$. Notice that the zeros of $Q_{n}(x),
n\geq 1,$ are real and simple (see \cite{MPP-RdM92}, Proposition 3.2). In
the next two tables we provide numerical evidence that supports this fact.
In the sequel, let $\left\{ \eta _{n,k}\right\} _{k=0}^{n}\equiv \eta
_{n,1}<\eta _{n,2}<...<\eta _{n,n}$ be the zeros of $Q_{n}(x)$ and $\left\{
x_{n,k}\right\} _{k=0}^{n}$ be the zeros of $F_{n}(x)$ arranged in an
increasing order. Next we show the position of the second zero of the
Freud-Sobolev-type polynomial of degree $n=4$ (namely $Q_{4}(x)$) and the
second and third zeros of $Q_{5}(x)$ for some choices of the masses $M_{0}$
and $M_{1}$. For $M_{0}=M_{1}=0$ we obviously recover the corresponding
zeros of the Freud polynomials. The first table shows the position of the
aforementioned zeros for $M_{0}=0$ and several values for $M_{1}$. The cases
when between the second and third (resp. third and fourth) zeros of $%
Q_{5}(x) $ there are no zeros of $Q_{4}(x),$ i.e. the zero interlacing for
the sequence $\{Q_{n}(x)\}_{n\geq0}$ fails, are shown in bold.


\begin{table}[!ht]
\centering\renewcommand{\arraystretch}{1.2} 
\begin{tabular}{rlrllr}
\hline
& $M_{0}=0.0$ &  &  &  &  \\ \cline{2-6}
& \multicolumn{1}{r}{$\eta _{5,2}$} & $\eta _{4,2}$ & \multicolumn{1}{r}{$%
\eta _{5,3}$} & \multicolumn{1}{r}{$\eta _{4,3}$} & $\eta _{5,4}$ \\ \hline
$M_{1}=0.0$ & \multicolumn{1}{r}{$-0.655248$} & $-0.39615$ & 
\multicolumn{1}{r}{$0.0$} & \multicolumn{1}{r}{$0.39615$} & $0.655248$ \\ 
$M_{1}=0.2$ & \multicolumn{1}{r}{$-0.458455$} & $-0.39615$ & 
\multicolumn{1}{r}{$0.0$} & \multicolumn{1}{r}{$0.39615$} & $0.458455$ \\ 
$M_{1}=0.4$ & \multicolumn{1}{r}{$\mathbf{-0.371898}$} & $-0.39615$ & 
\multicolumn{1}{r}{$0.0$} & \multicolumn{1}{r}{$0.39615$} & $\mathbf{0.371898%
}$ \\ 
$M_{1}=1.0$ & \multicolumn{1}{r}{$\mathbf{-0.261023}$} & $-0.39615$ & 
\multicolumn{1}{r}{$0.0$} & \multicolumn{1}{r}{$0.39615$} & $\mathbf{0.261023%
}$ \\ \hline
\end{tabular}%
\caption{Zeros of $Q_{5}(x)$ and $Q_{4}(x)$ for fixed $M_{0}=0.0$ and some
values of $M_{1}$.}
\label{Tabla1}
\end{table}





\begin{table}[!ht]
\centering\renewcommand{\arraystretch}{1.2} 
\begin{tabular}{rlrllr}
\hline
& $M_{0}=1.0$ &  &  &  &  \\ \cline{2-6}
& \multicolumn{1}{r}{$\eta _{5,2}$} & $\eta _{4,2}$ & \multicolumn{1}{r}{$%
\eta _{5,3}$} & \multicolumn{1}{r}{$\eta _{4,3}$} & $\eta _{5,4}$ \\ \hline
$M_{1}=0.0$ & \multicolumn{1}{r}{$-0.655248$} & $-0.284325$ & 
\multicolumn{1}{r}{$0.0$} & \multicolumn{1}{r}{$0.284325$} & $0.655248$ \\ 
$M_{1}=0.4$ & \multicolumn{1}{r}{$-0.371898$} & $-0.284325$ & 
\multicolumn{1}{r}{$0.0$} & \multicolumn{1}{r}{$0.284325$} & $0.371898$ \\ 
$M_{1}=0.9$ & \multicolumn{1}{r}{$\mathbf{-0.272822}$} & $-0.284325$ & 
\multicolumn{1}{r}{$0.0$} & \multicolumn{1}{r}{$0.284325$} & $\mathbf{%
0.272822}$ \\ 
$M_{1}=2.0$ & \multicolumn{1}{r}{$\mathbf{-0.192081}$} & $-0.284325$ & 
\multicolumn{1}{r}{$0.0$} & \multicolumn{1}{r}{$0.284325$} & $\mathbf{%
0.192081}$ \\ \hline
\end{tabular}%
\caption{Zeros of $Q_{5}(x)$ and $Q_{4}(x)$ for fixed $M_{0}=1.0$ and some
values of $M_{1}$.}
\label{Tabla2}
\end{table}


Observe that, as expected, the variation of $M_{1}$ only affects the
position of $\eta _{5,2}$ and $\eta _{5,4}$ and the variation of $M_{0}$
only affects the position of $\eta _{4,2}$ and $\eta _{4,4}$. This numerical
example is also reflected in Figure \ref{[S4]-FigRupture}.

\begin{figure}[!ht]
\centerline{\includegraphics[width=11cm,keepaspectratio]{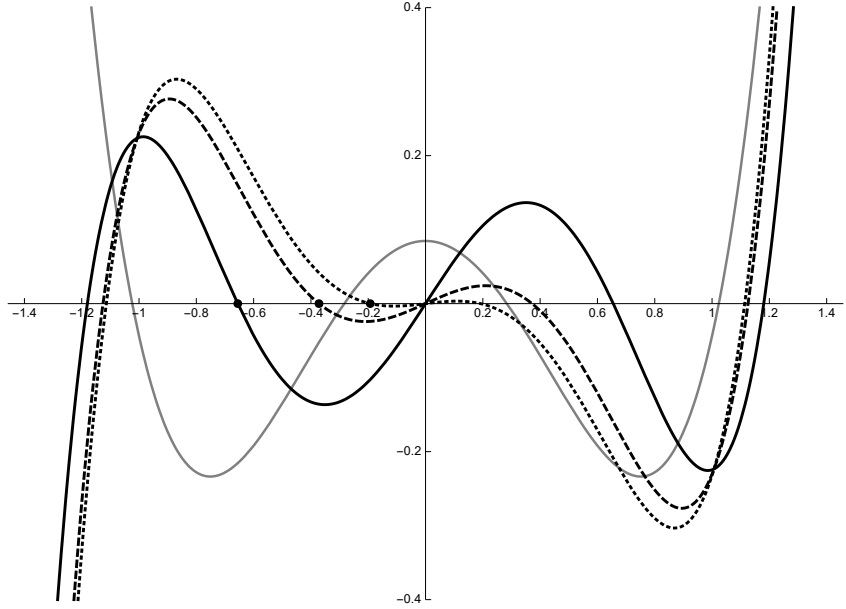}}
\caption{The figure shows, for a fixed value $M_{0}=1$, the evolution of the
second zero of the Freud-Sobolev type polynomial $Q_{5}(x)$ for three
different values of the mass $M_{1}$. The curve in gray color represents the
Freud-Sobolev type $Q_{4}(x)$, which is not affected by the variation of $%
M_{1}$. The zero of $Q_{5}(x;M_{1}=0)=F_{5}(x)$ (continuous black graph)
occurs at $\protect\eta _{5,2}(M_{1}=0)=-0.655248$. For $Q_{5}(x;M_{1}=0.2)$
(dashed line) we have $\protect\eta _{5,2}(M_{1}=0.2)=-0.371898$ and for $%
Q_{5}(x;M_{1}=2)$ (dotted line) occurs at $\protect\eta%
_{5,2}(M_{1}=2)=-0.19208$. Notice that for $M_{0}=1$ and $M_{1}=2$ there is
no zero of the polynomial $Q_{4}(x)$ between the second ($\protect\eta%
_{5,2}(M_{1}=2)=-0.19208$) and third ($\protect\eta_{5,3}(M_{1}=2)=0$) roots
of $Q_{5}(x;M_{1}=2)$, so the interlacing of the complete Freud-Sobolev type
orthogonal polynomial sequence $\{Q_{n}(x)\}_{n\geq 0}$ has been broken.}
\label{[S4]-FigRupture}
\end{figure}



\subsection{Asymptotic behavior}



We are interested in the dynamics of the zeros of the Freud-Sobolev type
when $M_{0}$ and $M_{1}$ tend, respectively, to infinity. To that end, let
us introduce the following the limit polynomials 
\begin{eqnarray}
G_{2n}(x) &=&\lim_{M_{0}\rightarrow \infty }Q_{2n}(x)=F_{2n}(x)-\frac{%
F_{2n}(0)}{K_{2n-2}(0,0)}K_{2n-2}(x,0),  \notag \\
J_{2n+1}(x) &=&\lim_{M_{1}\rightarrow \infty }Q_{2n+1}(x)=F_{2n+1}\left(
x\right) -\frac{[F_{2n+1}]^{\prime }(0)}{K_{2n-1}^{(1,1)}(0,0)}%
K_{2n-1}^{(0,1)}(x,0).  \label{JPoly-1}
\end{eqnarray}%
Similar polynomials have been previously studied in \cite{MPP-RdM92}, when
the discrete mass points are located outside the support of the perturbed
measure. Here, we find a slightly different situation because the support of
the measure is the whole real line and the discrete masses $M_{0}$ and $%
M_{1} $ are both located at $x=0\in \mathbb{R}$. As stated before, $M_{0}$
only affects the even degree polynomials, and the dynamics for the zeros of $%
\{Q_{2n}(x)\}_{n\geq 0}$ has been already obtained in \cite{AHM-AMC16}.
Next, we extend those results for the odd sequence $\{Q_{2n+1}(x)\}_{n\geq
0} $.

Our goal is to obtain results concerning the monotonicity and speed of
convergence of the zeros of\ $Q_{2n+1}(x)$. For this purpose we need the
following lemma concerning the behavior and the asymptotics of the zeros of
linear combinations of two polynomials with interlacing zeros, whose proof
we omit (see \cite[Lemma 1]{BDR-JCAM02} or \cite[Lemma 3]{DMR-ANM10}).

\begin{lemma}
\label{LEMMA-3} Let $\mathfrak{f}_{n}(x)=a(x-x_{1})\cdots (x-x_{n})$ and $%
\mathfrak{j}_{n}(x)=b(x-y_{1})\cdots (x-y_{n})$ be polynomials with real and
simple zeros, where $a$ and $b$ are positive real constants.

If%
\begin{equation*}
y_{1}<x_{1}<\cdots <y_{n}<x_{n},
\end{equation*}%
then, for any real constant $c>0$, the polynomial%
\begin{equation*}
\mathfrak{q}_{n}(x)=\mathfrak{f}_{n}(x)+c\mathfrak{j}_{n}(x)
\end{equation*}%
has $n$ real zeros $\eta _{1}<\cdots <\eta _{n}$ which interlace with the
zeros of $\mathfrak{f}_{n}(x)$ and $\mathfrak{j}_{n}(x)$ as follows%
\begin{equation*}
y_{1}<\eta _{1}<x_{1}<\cdots <y_{n}<\eta _{n}<x_{n}.
\end{equation*}%
Moreover, each $\eta _{k}=\eta _{k}(c)$ is a decreasing function of $c$ and,
for each $k=1,\ldots ,n$,%
\begin{equation*}
\lim_{c\rightarrow \infty }\eta _{k}=y_{k}\quad \text{and}\quad
\lim_{c\rightarrow \infty }c[\eta _{k}-y_{k}]=\dfrac{-\mathfrak{f}_{n}(y_{k})%
}{\mathfrak{j}_{n}^{\prime }(y_{k})}.
\end{equation*}
\end{lemma}

Before stating the main result of this Section, we will prove some auxiliary
results concerning the interlacing properties of $\{F_{2n+1}\}_{n\geq 0}$, $%
\{K_{2n-1}^{(0,1)}(x,0)\}_{n\geq 0}$, and $\{J_{2n+1}\}_{n\geq 0}$.

\begin{lemma}
The zeros of $\{K_{2n+1}^{(0,1)}(x,0)\}_{n\geq0}$, are real and simple.
Moreover, for every $n\geq1$, the non vanishing zeros of $%
K_{2n+1}^{(0,1)}(x,0)$ and $K_{2n-1}^{(0,1)}(x,0)$ interlace.
\end{lemma}

\begin{proof}
First, since $K_{2n-1}^{(0,1)}(x,0)$ is an odd polynomial, we can write $%
K_{2n+1}^{(0,1)}(x,0)=xs_{n}(x^2)$, where $s_{n}$ is a polynomial of degree $%
n$. We will prove that $\{s_n(y)\}_{n\geq 0}$, with $y=x^2$, is an
orthogonal polynomial sequence with respect to the measure $%
d\sigma(y)=y^{3/2}e^{-y^2}dy$, which is positive in the positive real line.
Indeed, for $n\neq m$, we have 
\begin{eqnarray*}
\int_0^\infty s_n(y)s_m(y)d\sigma(y)&=&\int_{-\infty}^{\infty}\frac{%
K_{2n+1}^{(0,1)}(x,0)}{x}\frac{K_{2m+1}^{(0,1)}(x,0)}{x}x^3e^{-x^4}(2xdx) \\
&=&2\int_{-\infty}^{%
\infty}K_{2n+1}^{(0,1)}(x,0)K_{2m+1}^{(0,1)}(x,0)x^2e^{-x^4}dx \\
&=&0,
\end{eqnarray*}
by using the reproducing property of $K_{2n-1}^{(0,1)}(x,0)$. On the other
hand, for $n=m$, and taking into account \eqref{[Sec1]-Knxy01der-P} and the
symmetry of the Freud polynomials, we get 
\begin{eqnarray*}
\int_0^\infty
s_n^2(y)d\sigma(y)&=&\int_{-\infty}^{%
\infty}K_{2n+1}^{(0,1)}(x,0)K_{2n+1}^{(0,1)}(x,0)x^2e^{-x^4}dx \\
&=&\int_{-\infty}^{\infty}K_{2n+1}^{(0,1)}(x,0)\frac{
xF_{2n+2}(x)[F_{2n+1}]^{\prime}(0)-F_{2n+1}(x)F_{2n+2}(0)}{\|F_{2n+1}\|^2}%
e^{-x^4}dx \\
&=&\frac{1}{\|F_{2n+1}\|^2}\left(([F_{2n-1}]^{\prime}(0))^2\|F_{2n+2}%
\|^2-[F_{2n+1}]^{\prime}(0)F_{2n+2}(0)\right)>0,
\end{eqnarray*}
since $[F_{2n+1}]^{\prime}(0)F_{2n+2}(0)<0$. As a consequence, the zeros of $%
s_n(x)$ are real, simple, and they are located in the positive real
semiaxis. Moreover, the zeros of $s_{n}(x)$ and $s_{n-1}(x)$ interlace. Now,
because of the symmetry, all polynomials of the sequence $%
\{K_{2n+1}^{(0,1)}(x,0)\}_{n\geq0} $ have a zero at the origin, and the
remaining zeros are located symmetrically at both sides of the origin.
Furthermore, if we denote by $s_{n,k}$ the $k$th zero of $s_n(x)$, then it
is clear from the definition that $\pm\sqrt{s_{n,k}}$ are zeros of $%
K_{2n+1}^{(0,1)}(x,0)$. As a consequence, the (non vanishing) zeros of $%
K_{2n+1}^{(0,1)}(x,0)$ and $K_{2n-1}^{(0,1)}(x,0)$ interlace.
\end{proof}

The next Lemma shows that the non vanishing zeros of $F_{2n+1}$ and $%
K_{2n-1}^{(0,1)}(x,0)$ also interlace.

\begin{lemma}
\label{Lema_zeros_F_K} Let $\{x_{2n+1,k}\}_{k=1}^{2n+1}$ and $%
\{z_{2n-1,k}\}_{k=1}^{2n-1}$ be the set of zeros of $F_{2n+1}$ and $%
K_{2n-1}^{(0,1)}(x,0)$, respectively, arranged in increasing order. Then, we
have%
\begin{eqnarray*}
x_{2n+1,k} &<&z_{2n-1,k}<x_{2n+1,k+1},\quad 1\leq k\leq n-1, \\
x_{2n+1,k+1} &<&z_{2n-1,k}<x_{2n+1,k+2}\quad n+1\leq k\leq 2n-1.
\end{eqnarray*}
\end{lemma}

\begin{proof}
Due to the symmetry of both polynomials, it suffices to prove the
interlacing for the positive zeros. Since $x_{2n+1,n+1}=z_{2n-1,n}=0$, we
consider the case when $n+1\leq k\leq 2n-1$. From \eqref{[Sec1]-Knxy01der-P}
and the symmetry of the Freud polynomials, we have 
\begin{eqnarray*}
x^{2}K_{2n-1}^{(0,1)}(x,0) &=&\frac{1}{\Vert F_{2n-1}\Vert ^{2}}\left(
xF_{2n}(x)[F_{2n-1}]^{\prime }(0)-F_{2n-1}(x)F_{2n}(0)\right) \\
&=&xF_{2n}(x)[F_{2n-1}]^{\prime }(0)-\left( \frac{xF_{2n}(x)-F_{2n+1}(x)}{%
a_{2n}^{2}}\right) F_{2n}(0),
\end{eqnarray*}%
where we have used (\ref{3TRR-Monic}) on the second equality. As a
consequence, evaluating the previous equation in $x_{2n+1,k+1}$ and $%
x_{2n+1,k+2}$ we obtain, respectively, 
\begin{eqnarray*}
x_{2n+1,k+1}^{2}K_{2n-1}^{(0,1)}(x_{2n+1,k+1},0)
&=&x_{2n+1,k+1}F_{2n}(x_{2n+1,k+1})\left( [F_{2n-1}]^{\prime }(0)-\frac{%
F_{2n}(0)}{a_{2n}^{2}}\right) , \\
x_{2n+1,k+2}^{2}K_{2n-1}^{(0,1)}(x_{2n+1,k+2},0)
&=&x_{2n+1,k+2}F_{2n}(x_{2n+1,k+2})\left( [F_{2n-1}]^{\prime }(0)-\frac{%
F_{2n}(0)}{a_{2n}^{2}}\right) .
\end{eqnarray*}%
Since $x_{2n+1,k+1}$ and $x_{2n+1,k+2}$ are positive and the zeros of the
Freud polynomials interlace, $F_{2n}(x_{2n+1,k+1})$ and $%
F_{2n}(x_{2n+1,k+2}) $ have distinct sign. As a consequence, $%
K_{2n-1}^{(0,1)}(x_{2n+1,k+1},0)$ and $K_{2n-1}^{(0,1)}(x_{2n+1,k+2},0)$
differ in sign, which means that $K_{2n-1}^{(0,1)}(x,0)$ has a zero between
the zeros $x_{2n+1,k+1}$ and $x_{2n+1,k+2}$.
\end{proof}

\begin{remark}
Notice that $F_{2n+1}$ and $K_{2n-1}^{(0,1)}(x,0)$ differ in two degrees.
This causes that the zeros interlacing between them is not complete. Indeed, 
$K_{2n-1}^{(0,1)}(x,0)$ has not zeros in the interval $[x_{2n+1,n},
x_{2n+1,n+2}]$, i.e. between the origin and the first zeros of $F_{2n+1}(x)$
at both sides.
\end{remark}

We will need some results concerning the interlacing properties of the zeros
of $F_{2n+1}(x)$, $J_{2n+1}(x)$ and $Q_{2n+1}(x)$. By symmetry, for the
zeros of $F_{2n+1}(x)$, we have $x_{2n+1,n+1}=0$ and $%
x_{2n+1,k}=-x_{2n+1,2n+2-k}$ for $1\leq k\leq n$. As a consequence, it
suffices to analyze the behavior of the positive zeros. In order to simplify
the notation, we denote $x_{k}:=x_{2n+1, n+1+k}$, $1\leq k \leq n$, i.e. $%
\{x_k\}_{k=1}^{n}$ are the $n$ positive zeros of $F_{2n+1}$ arranged in
increasing order. A similar notation will be used for the zeros of $Q_{2n+1}$
and $F_{2n+1}$. The following result is a straightforward corollary of Lemma %
\ref{Lema_zeros_F_K}.

\begin{corollary}
Let us denote by $\{y_{k}\}_{k=1}^{n}$ the set of positive zeros of $%
J_{2n+1}(x)$ arranged in increasing order. Then, for $1\leq k\leq n-1$, we
have%
\begin{equation}
x_{k}<y_{k+1}<x_{k+1},  \label{interlacingFK}
\end{equation}%
i.e., positive zeros of $J_{2n+1}(x)$ and $F_{2n+1}(x)$ interlace.
\end{corollary}

\begin{proof}
Taking into account the symmetry and the fact that $[J_{2n+1}]^{\prime}(0)=0$%
, we deduce that $J_{2n+1}(x)$ has a zero of multiplicity $3$ at the origin.
This is, $y_{1}=0$. The result follows by evaluating \eqref{JPoly-1} at two
consecutive zeros $x_{k}$ and $x_{k+1}$ of $F_{2n+1}$, for $1\leq k\leq n-1$%
, and noticing that by Lemma \ref{Lema_zeros_F_K}, $J_{2n+1}(x_{k})$ and $%
J_{2n+1}(x_{k+1})$ have distinct sign.
\end{proof}

\begin{remark}
Observe that due to the triple zero at the origin, $J_{2n+1}(x)$ does not
have a zero in the interval $(0, x_{1})$, i.e. between the origin and the
first positive zero of $F_{2n+1}(x)$. Since $F_{2n+1}(x)$ only have $n-1$
positive zeros, we have $y_1=0$.
\end{remark}



Now, we are ready to enunciate the main result of this Section.



\begin{theorem}
\label{InterlQKrall} On the positive real line, the following interlacing
property holds 
\begin{equation*}
0=y_{1}<\eta _{1}<x_{1}<y_{2}<\eta _{2}<x_{2}\cdots <y_{k}<\eta _{n}<x_{n}.
\end{equation*}%
Moreover, each $\eta _{k}:=\eta _{k}(M_{1})$ is a decreasing function of $%
M_{1}$ and, for each $k=1,\ldots ,n$,%
\begin{equation}
\lim_{M_{1}\rightarrow \infty }\eta _{k}(M_{1})=y_{k}\,,  \label{zeros_limit}
\end{equation}%
as well as%
\begin{equation}
\lim\limits_{M_{1}\rightarrow \infty }M_{1}[\eta _{k}(M_{1})-y_{k}]=\dfrac{%
-F_{2n+1}(y_{k})}{K_{2n-1}^{(1,1)}(0,0)[J_{2n+1}^{\prime }(y_{k})]}.
\label{convergence}
\end{equation}
\end{theorem}



\begin{proof}
Notice that the polynomials $\{\tilde{Q}_{2n+1}(x)\}_{n\geq 0}$ with $\tilde{%
Q}_{2n+1}(x)=\rho _{2n+1}Q_{2n+1}(x)$, can be represented as%
\begin{equation*}
\tilde{Q}_{2n+1}(x)=F_{2n+1}(x)+M_{1}K_{2n-1}^{(1,1)}\left( 0,0\right)
J_{2n+1}\left( x\right) ,  \label{LagSobNorm-2}
\end{equation*}%
where%
\begin{equation*}
\rho _{2n+1}=1+M_{1}K_{2n-1}^{(1,1)}\left( 0,0\right).
\end{equation*}
Thus, the interlacing follows at once from \eqref{interlacingFK} and Lemma %
\ref{LEMMA-3}. On the other hand, we can write 
\begin{equation*}
x\hat{q}_n(x^2)=x\hat{f}_n(x^2)+M_{1}K_{2n-1}^{(1,1)}(0,0)x\hat{j}_n(x^2),
\end{equation*}
with 
\begin{eqnarray*}
\hat{f}_n&=&(x-x_1^2)\cdots(x-x_n^2), \\
\hat{q}_n&=&(x-\eta_1^2)\cdots(x-\eta_n^2), \\
\hat{j}_n&=&(x-y_1^2)\cdots(x-y_n^2),
\end{eqnarray*}
and by the previous results, their zeros are real, simple and interlace, so
they satisfy the conditions on Lemma \ref{LEMMA-3}, and therefore 
\begin{equation*}
\lim_{M_1\rightarrow\infty}\eta_k^2=y_k^2,
\end{equation*}
and 
\begin{equation*}
\lim_{M_1\rightarrow\infty}=M_{1}K_{2n-1}^{(1,1)}(0,0)[\eta_k^2-y_k^2]=-%
\frac{\hat{f}_n(y_k^2)}{\hat{j}_n^{\prime}(y_k^2)}=-\frac{2y_k F_{2n+1}(y_k)%
}{[J_{2n+1}]^{\prime}(y_k)},
\end{equation*}
and since $\eta_k^2-y_k^2=(\eta_k+y_k)(\eta_k-y_k)$ and $\lim_{M_1%
\rightarrow\infty}\eta_k=y_k$, the result follows.
\end{proof}

\begin{remark}
Because of the symmetry, the limits \eqref{zeros_limit} and %
\eqref{convergence} also hold for the negative zeros. The only difference is
that those zeros are increasing functions of $M_1$.
\end{remark}

\begin{figure}[!ht]
\centerline{\includegraphics[width=11cm,keepaspectratio]{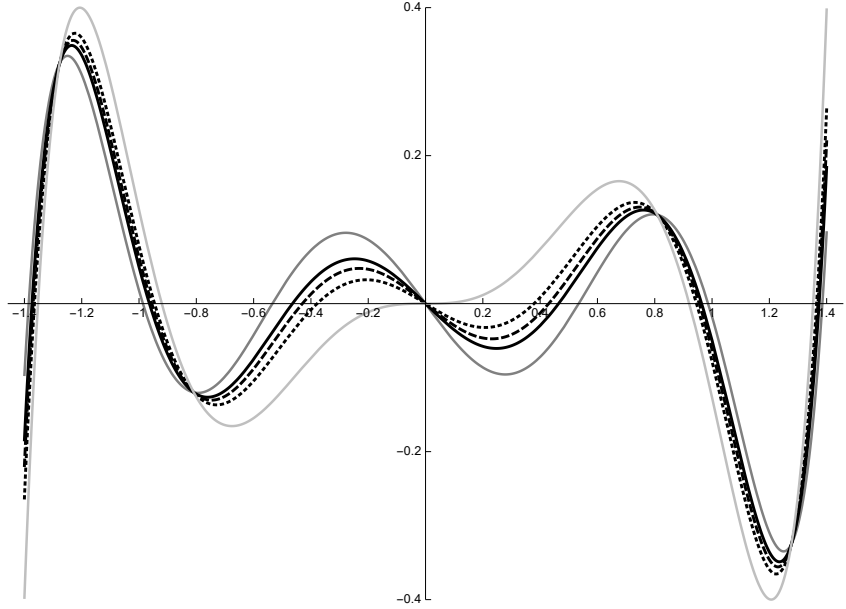}}
\caption{It illustrates the variation of the zeros of an odd degree
Freud-Sobolev type polynomials when $M_{1}$ varies as described in Theorem 
\protect\ref{InterlQKrall}. The graphs of $Q_{7}(x)$ for three different
values of $M_{1}$ are plotted. The black continuous, dashed, and dotted
lines correspond to $M_{1}=0.03$, $M_{1}=0.05$, and $M_{1}=0.09$,
respectively. We also include the graphs of $F_{7}(x)$ (medium gray color)
and $J_{7}(x)$ (light gray color), showing that the zeros of $Q_{7}(x)$ are
increasing functions of $M_{1}$ in the negative real semiaxis, traveling
from the negative zeros of $F_{7}(x)$ to the corresponding zeros of $%
J_{7}(x) $ as $M_{1}$ increases. Likewise, the positive zeros of $Q_{7}(x) $
are decreasing functions of $M_{1}$, traveling from the positive zero of $%
F_{7}(x)$ to the corresponding zero of $J_{7}(x)$ according with Theorem 
\protect\ref{InterlQKrall}. Observe that in this picture, the value of $%
M_{0} $ is irrelevant.}
\label{[S4]-FigLag}
\end{figure}



\section{Holonomic equation and electrostatic interpretation}



In this section, we deduce a second order differential equation satisfied by 
$\{Q_{n}(x)\}_{n\geq 0}$ and, as an application, an electrostatic
interpretation of its zeros is presented. We will use the connection formula
between $Q_{n}$ and $F_{n}$, which for convenience will take the form %
\eqref{conn_holo}. We will also use the structure formula \eqref{fprimDer}
(for the monic normalization) and the three term recurrence relation %
\eqref{3TRR-Monic}. Let us rewrite these formulas as 
\begin{eqnarray}
Q_{n}(x) &=&A_{n}(x)F_{n}(x)+B_{n}(x)F_{n-1}(z),  \label{connect} \\
F_{n}^{\prime }(x) &=&a_{n}(x)F_{n}(x)+b_{n}(x)F_{n-1}(x),  \label{struc} \\
F_{n+1}(x) &=&\beta _{n}(x)F_{n}(x)+\gamma _{n}(x)F_{n-1}(x),
\label{threeterm}
\end{eqnarray}%
where the coefficients above are given according to \eqref{conn_holo}, %
\eqref{fprimDer} and \eqref{3TRR-Monic}, respectively. Before stating our
main result, we need the following Lemmas.

\begin{lemma}
\label{[S3]-LEMA-1}The monic sequences $\{Q _{n}(x)\}_{n\geq 0}$ and $%
\{F_{n}(x)\}_{n\geq 0}$ satisfy%
\begin{equation}
\lbrack Q _{n}(x)]^{\prime }=C_{1}(x;n)F _{n}(x)+D_{1}(x;n)F_{n-1}(x)
\label{DerPSI-C1D1}
\end{equation}%
where%
\begin{eqnarray}
C_{1}(x;n) &=&A_{n}^{\prime }(x)+A_{n}(x)a(z;n)+B_{n}(x)\frac{b_{n-1}(x)}{%
\gamma_{n-1} (x)},  \label{[S3]-Coefs-ABCD1} \\
D_{1}(x;n) &=&B_{n}^{\prime }(x)+A_{n}(x)b_n(x)+B_{n}(x)\left( a_{n-1}(x)-%
\frac{\beta_{n-1} (x)}{\gamma_{n-1} (x)}\right) .  \notag
\end{eqnarray}
\end{lemma}



\begin{proof}
Notice that, combining \eqref{struc} and \eqref{threeterm} we have 
\begin{equation*}
\lbrack F_{n-1}(x)]^{\prime }=\frac{b_{n-1}(x)}{\gamma _{n-1}(x)}%
F_{n}(x)+\left( a_{n-1}(x)-\frac{\beta _{n-1}(x)}{\gamma _{n-1}(x)}\right)
F_{n-1}(x).
\end{equation*}%
The result follows by substituting the last equation and \eqref{struc} into
the derivative with respect to $x$ of \eqref{connect}. 
\end{proof}

\begin{lemma}
\label{[S3]-LEMA-4}The sequences of monic polynomials $\{Q_{n}(x)\}_{n\geq
0} $ and $\{F_{n}(x)\}_{n\geq 0}$ are also related by%
\begin{eqnarray}
{Q_{n-1}(x)} &=&A_{2}(x;n)F_{n}(x)+B_{2}(x;n)F_{n-1}(x),  \label{PSInm1-A2D2}
\\
\lbrack {Q}_{n-1}(x)]^{\prime } &=&C_{2}(x;n)F_{n}(x)+D_{2}(x;n)F_{n-1}(x),
\label{DzPSInm1-C2D2}
\end{eqnarray}%
where%
\begin{eqnarray*}
A_{2}(x;n) &=&\frac{B_{n-1}(x)}{\gamma _{n-1}(x)},\quad
B_{2}(x;n)=A_{n-1}(x)-B_{n-1}(x)\frac{\beta _{n-1}(x)}{\gamma _{n-1}(x)}, \\
C_{2}(x;n) &=&\frac{D_{1}(x;n-1)}{\gamma _{n-1}(x)},\quad
D_{2}(x;n)=C_{1}(x;n-1)-D_{1}(x;n-1)\frac{\beta _{n-1}(x)}{\gamma _{n-1}(x)}.
\end{eqnarray*}%
The coefficients $C_{1}(x;n-1)$ and $D_{1}(x;n-1)$ are given in (\ref%
{[S3]-Coefs-ABCD1}).
\end{lemma}



\begin{proof}
The expressions follow from \eqref{connect} and \eqref{DerPSI-C1D1},
respectively, after a shift in the degree, and using \eqref{threeterm} to
express both of them in terms of $F_n$ and $F_{n-1}$.
\end{proof}

\begin{lemma}
\label{[S3]-LEMA-5}The following "inverse connection" formulas hold. 
\begin{eqnarray}
F _{n}(x) &=&\frac{B_{2}(x;n)}{\Lambda (x;n)}Q _{n}(z)-\frac{B_{n}(x)}{%
\Lambda (x;n)}Q_{n-1}(x),  \label{InvR-PSIn} \\
F _{n-1}(z) &=&\frac{-A_{2}(x;n)}{\Lambda (x;n)}Q _{n}(x)+\frac{A_{n}(x)}{%
\Lambda (x;n)}Q _{n-1}(x),  \label{InvR-PSInm1}
\end{eqnarray}%
where 
\begin{equation*}
\Lambda (x;n)=A_{n}(x)B_{2}(x;n)-A_{2}(x;n)B_{n}(x).
\end{equation*}
\end{lemma}

\begin{proof}
The result follows by solving the linear system defined by \eqref{connect}
and \eqref{PSInm1-A2D2}.
\end{proof}

Now, we replace (\ref{InvR-PSIn}) and (\ref{InvR-PSInm1}) in (\ref%
{DerPSI-C1D1}) and (\ref{DzPSInm1-C2D2}), respectively, to obtain the ladder
equations%
\begin{eqnarray*}
\lbrack Q_{n}(z)]^{\prime } &=&\left[ \frac{C_{1}(x;n)B_{2}(x;n)}{\Lambda
(x;n)}-\frac{D_{1}(x;n)A_{2}(x;n)}{\Lambda (x;n)}\right] Q_{n}(x) \\
&&+\left[ \frac{A_{n}(x)D_{1}(x;n)}{\Lambda (x;n)}-\frac{C_{1}(x;n)B_{n}(x)}{%
\Lambda (x;n)}\right] Q_{n-1}(x),
\end{eqnarray*}%
\begin{eqnarray*}
\lbrack {Q}_{n-1}(x)]^{\prime } &=&\left[ \frac{C_{2}(x;n)B_{2}(x;n)}{%
\Lambda (x;n)}-\frac{A_{2}(x;n)D_{2}(x;n)}{\Lambda (x;n)}\right] Q_{n}(x) \\
&&+\left[ \frac{A_{n}(x)D_{2}(x;n)}{\Lambda (x;n)}-\frac{C_{2}(x;n)B_{n}(x)}{%
\Lambda (x;n)}\right] Q_{n-1}(x),
\end{eqnarray*}%
which can be written in the more compact way 
\begin{eqnarray}
(\Xi (x;n,2)I-D_{x})Q_{n}(x) &=&\Xi (x;n,1)Q_{n-1}(x),  \label{Ladder1Psi} \\
(\Theta (x;n,1)I+D_{x})Q_{n-1}(x) &=&\Theta (x;n,2)Q_{n}(x),
\label{Ladder2Psi}
\end{eqnarray}%
where $I$ and $D_{x}$ are the identity and $x$-derivative operators,
respectively, by defining the determinants 
\begin{eqnarray}
\Xi (x;n,k) &=&\frac{1}{\Lambda (x;n)}%
\begin{vmatrix}
C_{1}(x;n) & A_{k}(x;n) \\ 
D_{1}(x;n) & B_{k}(x;n)%
\end{vmatrix}%
,  \label{DifPSI1} \\
\Theta (x;n,k) &=&\frac{1}{\Lambda (x;n)}%
\begin{vmatrix}
C_{2}(x;n) & A_{k}(x;n) \\ 
D_{2}(x;n) & B_{k}(x;n)%
\end{vmatrix}%
,  \label{DifPSI2}
\end{eqnarray}%
for $k=1,2$, where $A_{1}(x;n):=A_{n}(x)$. As a consequence, we have the
following result.



\begin{theorem}
\label{[S1]-THEO-1}Let $\mathfrak{b}_{n}$ and $\mathfrak{b}_{n}^{\dag }$ be
the differential operators%
\begin{eqnarray*}
\mathfrak{b}_{n} &=&\Xi (x;n,2)I-D_{x}, \\
\mathfrak{b}_{n}^{\dag } &=&\Theta (x;n,1)I+D_{x}.
\end{eqnarray*}%
Then,%
\begin{eqnarray*}
\mathfrak{b}_{n}[Q _{n}(x)] &=&\Xi (x;n,1)Q _{n-1}(x), \\
\mathfrak{b}_{n}^{\dag }[Q _{n-1}(x)] &=&\Theta (x;n,2)Q _{n}(x),
\end{eqnarray*}%
where $\Xi (x;n,k)$\ and $\Theta (x;n,k)$ are given in (\ref{DifPSI1}) and (%
\ref{DifPSI2}), respectively.
\end{theorem}

Finally, we state the main result of this section.

\begin{theorem}
\label{[S2]-THEO-5} The Sobolev-Freud type polynomials $\{Q_{n}(x)\}_{n\geq
0}$\ satisfy the second order linear differential equation%
\begin{equation}
\lbrack Q _{n}(x)]^{\prime \prime }+\mathcal{R}(x;n)[Q _{n}(x)]^{\prime }+%
\mathcal{S}(x;n)Q _{n}(x)=0,  \label{[S2]-2ndODE}
\end{equation}%
where%
\begin{eqnarray*}
\mathcal{R}(x;n) &=&\Theta (x;n,1)-\Xi (x;n,2)-\frac{[\Xi (x;n,1)]^{\prime }%
}{\Xi (x;n,1)}, \\
\mathcal{S}(x;n) &=&\Xi (x;n,2)\left[ \frac{[\Xi (x;n,1)]^{\prime }}{\Xi
(x;n,1)}-\Theta (x;n,1)\right] -[\Xi (x;n,2)]^{\prime }.
\end{eqnarray*}
\end{theorem}



\begin{proof}
The result follows in a straightforward way from the ladder operators
provided in Theorem \ref{[S1]-THEO-1}. The usual technique (see, for example 
\cite[Th. 3.2.3]{Ism05}) consists in applying the raising operator to both
sides of the equation satisfied by the lowering operator, i.e.%
\begin{equation*}
\mathfrak{b}_{n}^{\dag }\left[ \frac{1}{\Xi (x;n,1)}\mathfrak{b}_{n}[Q
_{n}(x)]\right] =\mathfrak{b}_{n}^{\dag }[Q _{n-1}(x)]=\Theta (x;n,2)Q
_{n}(x),
\end{equation*}%
and then using the definition $\mathfrak{b}_{n}^{\dag }$ to compute the left
hand side. After some computations, \eqref{[S2]-2ndODE} follows. 
%
\end{proof}

We point out that we have obtained a second order linear differential
equation for the complete sequence $\{Q_{n}(x)\}_{n\geq 0}$. However, as we
have mentioned in the previous sections, the even and odd degree polynomials
behave differently. Indeed, they have another connection formula, and the
previous results hold in either case just by taking the coefficients of the
connection formula \eqref{connect} accordingly. Using Mathematica$%
^{\circledR }$, the expression for $\mathcal{R}(x;n)$ was obtained according
to Theorem \ref{[S2]-THEO-5}. In the sequel, we provide the expressions for
the odd case ($\kappa _{n}^{[0]}=0,\kappa _{2n}^{[1]}=0,r_{2n+1}=1,r_{2n}=0$%
), together with an electrostatic interpretation of the zeros of $%
\{Q_{n}(x)\}_{n\geq 0}$. The even case was analyzed in \cite{AHM-AMC16}. We
found 
\begin{equation*}
\mathcal{R}(x;2n+1)=\frac{2}{x}-4x^{3}-\frac{u^{\prime }(x;2n+1)}{u(x;2n+1)},
\end{equation*}%
where $u(x;2n+1)$ is the biquartic polynomial%
\begin{equation}
u(x;2n+1)=u_{4}(n)\,x^{4}+u_{2}(n)\,x^{2}+u_{0}(n)  \label{ubiquartic}
\end{equation}%
with%
\begin{eqnarray*}
u_{4}(n) &=&16\phi _{2n+1}^{2}(0)[1+\kappa _{2n+1}^{[1]}], \\
u_{2}(n) &=&4\phi _{2n+1}(0)\left[ 4\phi _{2n+1}^{2}(0)+\kappa
_{2n+1}^{[1]}(2+\kappa _{2n+1}^{[1]})(4a_{2n+1}^{2}\phi _{2n+1}(0)-1)\right]
, \\
u_{0}(n) &=&\kappa _{2n+1}^{[1]}\left[ -12\phi _{2n+1}^{2}(0)+\kappa
_{2n+1}^{[1]}\left\{ 1+8a_{2n+1}^{2}\phi _{2n+1}(0)\left[ -1+2\phi
_{2n}(0)\phi _{2n+1}(0)\right] \right\} \right] .
\end{eqnarray*}%
Now, the evaluation of \eqref{[S2]-2ndODE} at the zeros $\{y_{2n+1,k}%
\}_{k=1}^{2n+1}$ of $Q_{2n+1}$ yields 
\begin{equation*}
\frac{\lbrack Q_{2n+1}]^{\prime \prime }(y_{2n+1,k})}{[Q_{2n+1}]^{\prime
}(y_{2n+1,k})}=-\mathcal{R}(y_{2n+1,k};2n+1)=-\frac{2}{y_{2n+1,k}}%
+4(y_{2n+1,k})^{3}+\frac{u^{\prime }(y_{2n+1,k};2n+1)}{u(y_{2n+1,k};2n+1)}.
\end{equation*}%
The above equation represents the electrostatic equilibrium condition for
the zeros $\{y_{2n+1,k}\}_{k=1}^{2n+1}$ of $Q_{2n+1}$ and can be rewritten
as (see \cite{Ism05} ) 
\begin{equation*}
\sum_{j=1,j\neq k}^{2n+1}\frac{1}{y_{2n+1,j}-y_{2n+1,k}}+\frac{u^{\prime
}(y_{2n+1,k};2n+1)}{2u(y_{2n+1,k};2n+1)}-\frac{1}{y_{2n+1,k}}%
+2(y_{2n+1,k})^{3}=0.
\end{equation*}%
Therefore, the zeros of $Q_{2n+1}$ are critical points of the total energy.
Thus, the electrostatic interpretation of the distribution of zeros means
that we have an equilibrium position under the action of the external
potential 
\begin{equation*}
V_{ext}(x,2n+1)=\frac{1}{2}\ln u(x;2n+1)-\frac{1}{2}\ln x^{2}e^{-x^{4}},
\end{equation*}%
where the first term represents a short range potential which corresponds to
unit charges located at the four zeros of $u(x;2n+1)$, and the second term
is a long range potential associated with a Christoffel perturbation of the
Freud weight function.

If $z_{+}(n)$ and $z_{-}(n)$ are the solutions of the associated quadratic
equation%
\begin{equation}
u_{4}(n)\,z^{2}+u_{2}(n)\,z+u_{0}(n)=0,  \label{quadratic}
\end{equation}%
then the zeros of (\ref{ubiquartic}) are%
\begin{equation*}
x_{1}(n)=+\sqrt{z_{+}}(n),\text{ }x_{2}(n)=-\sqrt{z_{+}}(n),\text{ }%
x_{3}(n)=+\sqrt{z_{-}}(n),\text{ }x_{4}(n)=-\sqrt{z_{-}}(n).
\end{equation*}



\begin{table}[!ht]
\centering\renewcommand{\arraystretch}{1.2} 
\begin{tabular}{lrrrrrrrr}
\hline
& \multicolumn{2}{c}{$M=0.1$} &  & \multicolumn{2}{c}{$M=1$} &  & 
\multicolumn{2}{c}{$M=10$} \\ \cline{2-3}\cline{5-6}\cline{8-9}
& $\pm \sqrt{z_{1}}$ & $\pm \sqrt{z_{2}}$ &  & $\pm \sqrt{z_{1}}$ & $\pm 
\sqrt{z_{2}}$ &  & $\pm \sqrt{z_{1}}$ & $\pm \sqrt{z_{2}}$ \\ \hline
$n=1$ & $\pm \,0.369164$ & $\pm \,0.878731\,i$ &  & $\pm \,0.745497$ & $\pm
\,0.914759\,i$ &  & $\pm \,0.905303$ & $\pm \,0.928589\,i$ \\ 
$n=3$ & $\pm \,0.397067$ & $\pm \,1.059517\,i$ &  & $\pm \,0.387740$ & $\pm
\,1.089036\,i$ &  & $\pm \,0.159258$ & $\pm \,1.106825\,i$ \\ 
$n=5$ & $\pm \,0.329766$ & $\pm \,1.181451\,i$ &  & $\pm \,0.197206$ & $\pm
\,1.197172\,i$ &  & $\pm \,0.068685$ & $\pm \,1.201241\,i$ \\ 
$n=7$ & $\pm \,0.251172$ & $\pm \,1.272375\,i$ &  & $\pm \,0.116257$ & $\pm
\,1.279623\,i$ &  & $\pm \,0.038576$ & $\pm \,1.280856\,i$ \\ 
$n=9$ & $\pm \,0.189032$ & $\pm \,1.345977\,i$ &  & $\pm \,0.076318$ & $\pm
\,1.349456\,i$ &  & $\pm \,0.024825$ & $\pm \,1.349937\,i$ \\ 
$n=11$ & $\pm \,0.144418$ & $\pm \,1.408813\,i$ &  & $\pm \,0.053943$ & $\pm
\,1.410616\,i$ &  & $\pm \,0.017374$ & $\pm \,1.410839\,i$ \\ 
$n=13$ & $\pm \,0.112816$ & $\pm \,1.464184\,i$ &  & $\pm \,0.040222$ & $\pm
\,1.465192\,i$ &  & $\pm \,0.012969$ & $\pm \,1.465308\,i$ \\ 
$n=15$ & $\pm \,0.089745$ & $\pm \,1.513969\,i$ &  & $\pm \,0.029902$ & $\pm
\,1.514571\,i$ &  & $\pm \,0.004691$ & $\pm \,1.514637\,i$ \\ 
$n=17$ & $\pm \,0.073204$ & $\pm \,1.559345\,i$ &  & $\pm \,0.024134$ & $\pm
\,1.559723\,i$ &  & $\pm \,0.005169$ & $\pm \,1.559764\,i$ \\ 
$n=19$ & $\pm \,0.060787$ & $\pm \,1.601140\,i$ &  & $\pm \,0.019950$ & $\pm
\,1.601389\,i$ &  & $\pm \,0.005144$ & $\pm \,1.601416\,i$ \\ 
\lasthline & \multicolumn{1}{l}{} &  &  &  &  &  &  & 
\end{tabular}%
\caption{Zeros of $u(x;2n+1)$ for several values of $M_{1}$ and odd values
of $n$.}
\label{Tabla3}
\end{table}



\begin{figure}[!ht]
\centerline{\includegraphics[width=10cm,keepaspectratio]{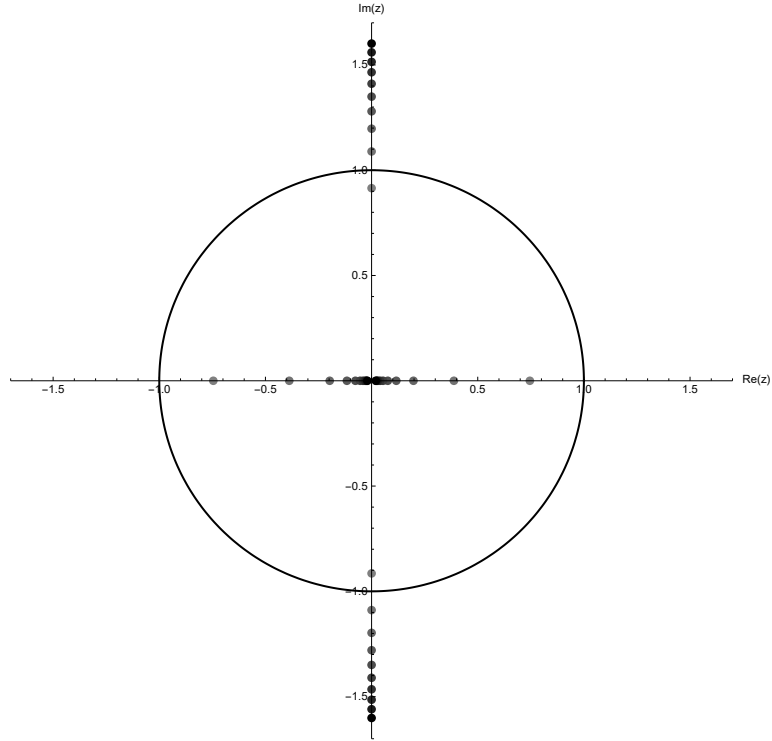}}
\caption{Zeros of $u(x;2n+1)$ for $M_{1}=1$ and odd values of $n$, from $1$
to $19$.}
\label{Grafica2.eps}
\end{figure}



Table \ref{Tabla3} shows the zeros of $u(x;2n+1)$ for some fixed values of $%
M_{1}$ and several values of $n$. With just a little more effort, we can
describe the asymptotic behavior with $n$ of these four roots. From Lemma %
\ref{[Sec2]-LEMMA2}, and (\ref{[Sec1]-LewQuarles}), after some tedious but
straightforward computations, the asymptotic behavior of the three
coefficients is%
\begin{eqnarray*}
u_{4}(n) &=&\frac{32}{3}n\left( 1+\frac{15}{8n}+\frac{155}{128n^{2}}+%
\mathcal{O}(n^{-3})\right) , \\
u_{2}(n) &=&8\sqrt{6}n^{1/2}\left( 1+\frac{4n}{9}+\frac{287}{432n}+\mathcal{O%
}(n^{-2})\right) , \\
u_{0}(n) &=&\frac{-9}{2}\left( 1+\frac{9}{8n}+\mathcal{O}(n^{-2})\right).
\end{eqnarray*}%
Then, the asymptotic behavior of the aforementioned $z_{+}$ and $z_{-}$ is%
\begin{eqnarray*}
z_{+}(n) &=&\frac{27}{64}\sqrt{\frac{3}{2}}\frac{1}{n^{3/2}}-\frac{243}{512}%
\sqrt{\frac{3}{2}}\frac{1}{n^{5/2}}+\mathcal{O}(n^{-7/2}), \\
z_{-}(n) &=&-\sqrt{\frac{2}{3}}\frac{1}{n^{1/2}}-\frac{1}{4}\sqrt{\frac{3}{2}%
}\frac{1}{n^{5/2}}+\mathcal{O}(n^{-7/2})
\end{eqnarray*}%
The above shows that, as $n$ goes to infinity, $z_{+}(n)$ remains positive and $z_{-}(n)$ negative, so $u(x;2n+1)$ will always have two symmetric real zeros $x_{1},x_{2}=\pm \sqrt{z_{+}}(n)$ , and two extra simple conjugate pure
imaginary zeros $x_{3},x_{4}=\pm \sqrt{z_{-}}(n)$.








\end{document}